\documentclass[12pt,a4paper]{amsart}

\usepackage[ruled,vlined,boxed,norelsize]{algorithm2e}

\usepackage{fullpage}
\usepackage{amsmath, amssymb, mathrsfs}
\usepackage{verbatim}
\usepackage{amsfonts}
\usepackage{hvfloat}
\usepackage{amscd}
\usepackage{blindtext}
\usepackage{fancyhdr}
 \usepackage{pdflscape}
\usepackage{rotating}
\usepackage{floatpag}
\usepackage{enumerate}
\usepackage{graphicx}
\usepackage{amsthm}
\usepackage{subcaption}
\usepackage[T1]{fontenc}
\usepackage[french, english]{babel}
\usepackage[T1]{fontenc}
\usepackage{chapterbib}
\usepackage[all]{xy}
\usepackage{smartdiagram}
\usepackage{url}
\usepackage{hyperref}
\usepackage{float}
\usepackage{multirow}
\usepackage{dsfont}
\usepackage{vmargin}
\usepackage{color}
\definecolor{gris}{gray}{0.45}
\usepackage{tikz}
\usetikzlibrary{arrows}
\usetikzlibrary{decorations.pathreplacing}
\usetikzlibrary{matrix}
\usetikzlibrary{calc}
\usetikzlibrary{shapes}
\usetikzlibrary{patterns}
\usetikzlibrary{fit,backgrounds,scopes}
\usepackage[noend]{algorithmic} % for the algorithmic enviroment

\newtheorem{theorem}{Theorem}[section]
\newtheorem{lemma}[theorem]{Lemma}
\newtheorem{proposition}[theorem]{Proposition}
\newtheorem{corollary}[theorem]{Corollary}
\usepackage{tikz}
\usetikzlibrary{positioning}
\usepackage{verbatim}

\newtheorem{conjecture}[theorem]{Conjecture}

\newtheorem{predefinition}[theorem]{Definition}
\newenvironment{definition}{\begin{predefinition}\rm}{\end{predefinition}}
\newtheorem{preremark}[theorem]{Remark}
\newenvironment{remark}{\begin{preremark}\rm}{\end{preremark}}
\newtheorem{prenotation}[theorem]{Notation}

\newtheorem{preexample}[theorem]{Example}
\newenvironment{example}{\begin{preexample}\rm}{\end{preexample}}
\newtheorem{preclaim}[theorem]{Claim}

\newtheorem{prequestion}[theorem]{Question}

 \makeatletter

\DeclareMathOperator{\Char}{Char}

\DeclareMathOperator{\Pic}{Pic}
\DeclareMathOperator{\Jac}{Jac}

\DeclareMathOperator{\Div}{Div}

\DeclareMathOperator{\Mult}{mult}

\newcommand{\FF}{\mathbb{F}}
\newcommand{\CC}{\mathbb{C}}
\newcommand{\HH}{\mathbb{H}}

\newcommand{\PP}{\mathbb{P}}

\newcommand{\ZZ}{\mathbb{Z}}

\newcommand{\Mod}[1]{\ (\mathrm{mod}\ #1)}

\newcommand{\calC}{\mathcal{C}}

\newcommand{\calL}{\mathcal{L}}

\newcommand{\calQ}{\mathcal{Q}}

\newcommand{\calS}{\mathcal{S}}

\def\Z{\mathbb{Z}}
\def\C{\mathbb{C}}

\def\QV{\textrm{QV}}
\def\QVO{\textrm{QV$_{-}$}}
\def\QVE{\textrm{QV$_{+}$}}

\def\eps{\epsilon}

\newcommand{\Ch}[2]{\begin{bmatrix} #1 \\ #2 \end{bmatrix}}
\def\ex{\boldsymbol{e}}

\newcommand{\car}[2]{%
\left[%
\begin{smallmatrix}%
\displaystyle#1\\%
\displaystyle#2\end{smallmatrix}\right]}

\title{A ThomAe-like Formula: \\Algebraic Computations of Theta Constants}
\author{Turku Ozlum Celik}
\begin{document}

%\title{An Algorithmic Generalisation of Weber's Formula}

%\author{}
\address{Max Planck Institute for Mathematics in the Sciences, InselStrasse 22, 04103, Leipzig, germany}
% \curraddr{}
\email{tuerkue.celik@mis.mpg.de}
% \thanks{}
%\subjclass[2010]{14Q05 (primary), 14H50 (secondary)}
%\subjclass[2010]{14H42 (primary), 14Q05 (secondary)}
%\keywords{algebraic curves, del Pezzo surfaces, theta characteristics, theta constants}

%\date{\today}

\maketitle
\begin{abstract}

We give an algebraic method to compute the fourth power of the quotient of any even theta constants associated to a given non-hyperelliptic curve in terms of geometry of the curve.
In order to apply the method, we work out non-hyperelliptic curves of genus 4, in particular, such curves lying on a singular quadric, which arise from del Pezzo surfaces of degree 1. Indeed, we obtain a complete 2-level structure of the curves by studying their theta characteristic divisors via exceptional divisors of the del Pezzo surfaces as the structure is required for the method.

\end{abstract}

\section{Introduction}

Computations of theta constants are closely related to a classical problem that asks which complex principally polarized abelian varieties arise as Jacobian varieties of curves. The problem is called the {\it Schottky problem} and goes back to Riemann \cite{Rie1857, Rie1866, FarGruMan2017}. 
In addition, the topic has many applications in different areas such as theoretical physics 
\cite{BobKle} 
via integrable systems, and cryptography \cite{Wen2003} via AGM-style point counting algorithms \cite{Rit2004} and isogeny based cryptography \cite{LubRob2016}. 

Let $g\geq 0$ be an integer. Denote $\mathcal{M}_g $ the moduli space over $\CC$ of curves of genus $g$ and $\mathcal{A}_g$ the moduli space of complex principally polarized abelian varieties of dimension $g$. The {\it Torelli map} 
$j:\mathcal{M}_g \rightarrow \mathcal{A}_g$
maps the isomorphism class of a curve to the isomorphism class of its Jacobian with its canonical polarization. The Schottky problem is to characterize the locus of Jacobians $\mathcal{J}_g$ which is defined to be the closure of $j(\mathcal{M}_g)$ in $\mathcal{A}_g$. Mumford showed that a principally polarized abelian variety can be written as an intersection of explicit quadrics in a projective space \cite{Mum1966}.  The coefficients of these quadrics are determined by theta constants denoted by $\vartheta[q](\tau)$, where $\tau $ is a \emph{Riemann matrix} for a specific choice of bases of regular
differentials and homology and $[q]\in \FF_2^g\oplus \FF_2^g$ is a \emph{characteristic}. Thomae-like formulas express these theta constants in terms of geometry of the curve. So the formulas can be seen as an explicit description of the Torelli map. In the case of
a hyperelliptic curve given by $y^2=\prod_{i=1}^{2g+2} (x-\alpha_i)$, we have 
$$\vartheta[q](\tau)^4 = (2 i \pi)^{-2g} \cdot \det(\Omega_1)^2 \cdot \prod_{i,j \in U} (\alpha_i-\alpha_j),$$
where $\Omega_1$ is the first half of a period matrix and $U$ is a set of indices depending on the
characteristic $[q]$ \cite[Page 218]{Tho1870}.

Let $\mathcal{C}$ be a non-hyperelliptic curve of genus $g$ over a field $k\subseteq \mathbb{C}$ and $\tau$ be a fixed period matrix. When $\calC$ is of genus 3, for any two even theta characteristics $p_1,p_2$ we have
\begin{equation}\label{WeberEng}
\left(\frac{\vartheta[p_1](\tau)}{\vartheta[p_2](\tau)}\right)^4=
(-1)^n\cdot 
\frac{\left[\beta_1,\beta_2,\beta_3\right] \cdot
\left[\beta_1,\beta_{12},\beta_{13}\right]
 \cdot \left[\beta_{12},\beta_{2},\beta_{23}\right] \cdot
\left[\beta_{13},\beta_{23},\beta_{3}\right]}{\left[
\beta_{23},\beta_{13},\beta_{12}\right] \cdot
\left[\beta_{23},\beta_{3},\beta_{2}\right] \cdot
\left[\beta_{3},\beta_{13},\beta_{1}\right] \cdot
\left[\beta_{2},\beta_{1},\beta_{12}\right]}, %\tag{$\star$}
\end{equation}
where $[\beta_i, \beta_{j},\beta_k]$ is the determinant of the
coefficients of the equations $\beta_i,\beta_j$ and $\beta_k$ which are certain lines labeled via some combinatorial data with respect to $\calC$ \cite[page 162]{Web1876} and $n=0,1$ can be computed depending on $p_1,p_2$. We call this Thomae-like formula {\it Weber's formula}. In this article, we mainly present a generalization of Weber's formula for any genus by getting motivated from \cite[Remark 1]{NarRit2017}. 

\begin{theorem}\label{mainTheorem} Let $A_i$'s and $B_i$'s be fixed representatives for the contact points of $\calC$ with specific hyperplanes for $i=1,\dots , g-1$. For any two even characteristics $p_1,p_2$, we have the quotients of explicitly computable homogeneous quadratic forms $\tilde{Q}^r_{i}$ and $\tilde{Q}^s_{i}$ in $g$ variables such that 

\begin{tikzpicture}
\node[anchor=center, scale=0.035cm] (eqn) {$(-1)^n \cdot  \frac{\vartheta[p_1](\tau)^4}{\vartheta[p_2](\tau)^4}
=\frac
{d_{1}
{\begin{vmatrix} 
\tilde{Q}^s_{1}(B_1)  & \cdots & \tilde{Q}^s_{g-1}(B_1) \\
\vdots & & \vdots \\
\tilde{Q}^s_{1}(B_{g-1})   & \cdots & \tilde{Q}^s_{g-1}(B_{g-1})
\end{vmatrix}}^2
{\begin{vmatrix} \tilde{Q}^r_{1}(A_1)  & \cdots & \tilde{Q}^s_{g-1}(A_1) \\ 
\vdots & & \vdots  \\ 
\tilde{Q}^r_{1}(A_{g-1})  & \cdots & \tilde{Q}^r_{g-1}(A_{g-1})
\end{vmatrix}}^2}
{d_{2}{\begin{vmatrix} \tilde{Q}^r_{1}(B_1)  & \cdots & \tilde{Q}^r_{g-1}(B_1) \\ 
\vdots & & \vdots  \\ 
\tilde{Q}^r_{1}(B_{g-1})  & \cdots & \tilde{Q}^r_{g-1}(B_{g-1})
\end{vmatrix}}^2
{\begin{vmatrix} 
\tilde{Q}^s_{1}(A_1)  & \cdots & \tilde{Q}^s_{g-1}(A_1) \\
\vdots & & \vdots \\
\tilde{Q}^s_{1}(A_{g-1})   & \cdots & \tilde{Q}^s_{g-1}(A_{g-1})
\end{vmatrix}}^2},$};
\end{tikzpicture}
where $d_1,d_2$ are the values of products of linear forms defining certain hyperplanes at the points $A_i$ and $B_i$'s and $n=0,1$ is given purely in terms of $p_1,p_2$. 

\end{theorem}

 We define the {\it Jacobian} of $\mathcal{C}$ as the quotient $\mathbb{C}^g/(\mathbb{Z}^g + \tau \mathbb{Z}^g)$ with respect to a {\it normalized period matrix} $\tau $ in the Siegel upper half space $\mathbb{H}_g$. Denote it $\textrm{Jac}(\mathcal{C})$. 
A {\it complete $2$-level structure of $\calC$} is represented via the defining equation(s) of the image of $\calC$ under the canonical embedding and certain divisors on the curve with a suitable labeling as follows. Such divisors are called {\it theta characteristic divisors}. There are two kinds of theta characteristic divisors, {\it even} and {\it odd}, which are determined by the parities of the dimensions of the associated Riemann-Roch spaces. The odd theta characteristic divisors correspond to some geometric objects called {\it multitangents} on the canonical model of $\mathcal{C}$. For instance, these objects are called {\it bitangents} when $g=3$ and {\it tritangents} when $g=4$.  On the other hand, there is a canonical correspondence between the theta characteristic divisors on $\mathcal{C}$ and the quadratic forms on $\textrm{Jac}(\mathcal{C})$[2] over $\mathbb{F}_2$, where $\textrm{Jac}(\mathcal{C})[2]$ denotes the 2-torsion subgroup of $\textrm{Jac}(\mathcal{C})$. The quadratic forms are {\it labeled} through some combinatorial data which is called an {\it Aronhold basis}. For instance, the labeling which appears in Weber's formula is due to an Aronhold basis. The complete 2-level structure plays a role for the generalization of Weber's formula.

The paper is organized as follows. Section~\ref{background} mainly recalls some mathematical background as a base of the sequel, such as quadratic forms over $\FF_2$, theta functions and characteristics, theta characteristic divisors, multitangents and their relations. We discuss the part reviewing quadratic forms in a coordinate-free setting peculiarly. We prove a coordinate-free version of \cite[Theorem A1.1]{RauFar1974} which is about obtaining an Aronhold basis. At the same time, we give another way of the labeling in terms of coordinates by using \emph{Steiner sets}. Section~\ref{background} might be considered also as a compact collection about the geometric, algebraic, combinatorial structures of non-hyperelliptic curves of genus g and the link among them. 

In Section~\ref{ALG}, we prove Theorem~\ref{mainTheorem}. As a consequence of the theorem, we give Algorithm~\ref{AlgTheCon} to compute the fourth power of the quotient of even theta constants associated to any given non-hyperelliptic curve.%, see Algorithm~\ref{AlgTheCon}.

In Section~\ref{DelPezzoSubsections}, in order to apply the algorithm, we study particularly curves of genus 4 since we have Weber's formula for the case of genus 3. Especially the ones which arise from del Pezzo surfaces of degree 1 are worked out because their geometric, algebraic and combinatorial properties are more accessible. We present a way to obtain a complete 2-level structure of these curves by finding an Aronhold basis thanks to geometry of the surface via some results in Section~\ref{background}. We exhibit Example~\ref{Ex} to apply the way and finally use the example for an explicit computation with our algorithm. 

\subsection*{Acknowledgments}
This article is based on a part of the doctoral dissertation of the author which is studied under the supervision of Christophe Ritzenthaler. She would like to thank him not only for the suggestion of the project but also for his support and valuable guidance.  The author is also thankful to Alessio Fiorentino, Avinash Kulkarni, Yue Ren and Mahsa Sayyary Namin for fruitful discussions about the topic and to Bernd Sturmfels for his beneficial remarks on the first version. 

\section{Theta Characteristics}\label{background}

In this section, we review some basic definitions and results about quadratic forms over $\FF_2$, theta functions and characteristics, theta characteristic divisors, multitangents and their relations devoted to the proof of Theorem~\ref{mainTheorem}. We refer to the classical source \cite{Dol2012} for this part. 

\subsection{Quadratic forms over $\FF_2$}

Let $g \geq 1$ be an integer and $V$ be a vector space of dimension
$2g$ over $\FF_2$. We fix a bilinear, non-degenerate, alternating form $\langle ,\rangle$ on $V$. Since $\Char \FF_2 =2$, there exists a basis $\{e_1,\dots , e_g, f_1 , \dots , f_g\}$ such that the matrix $\mathcal{M}_g$ associated to the bilinear form $\langle ,\rangle$ is 
\begin{align*}
\mathcal{M}_g =
 \begin{pmatrix} 0_g & I_g \\
I_g & 0_g 
\end{pmatrix},
\end{align*}
where $0_g, I_g$ are the zero, identity $g\times g$ matrices respectively. In other words, 
$$
\langle e_i,e_j\rangle=\langle f_i,f_j\rangle=0 \text{ and } \langle e_i,f_j\rangle=\delta_{ij} 
$$
for all $1\leq i,j\leq g$. 
Such a basis is called a {\it symplectic basis}.

We say that $q : V \to \FF_2$ is a {\it quadratic form}
on $V$ if 
$q(u+v)=q(u)+q(v)+\langle u,v\rangle$
for all $u,v \in V.$ Let $\QV$ denote the set of all quadratic forms on $(V,\langle , \rangle)$. 

The vector space $V$ has an action on $\QV$. Indeed, we define the quadratic form $q+v$ by 
$$(q+v)(u)=q(u)+\langle v,u \rangle$$ for any $q\in QV$ and $v\in V$. Since the form $\langle , \rangle$ is nondegenerate, the action is free. The equality $\# V = \# \QV$ implies that the action is also transitive. So, for any two quadratic forms $q,q'\in\QV$, there is a unique vector $v=q+q'$ such that 
$$\langle v,u \rangle=q(u)+q'(u).$$ In other words, the space $\QV$ is a homogeneous space for $V$. This implies that the disjoint union $$V\bigsqcup \QV$$ is an $\FF_2$-vector space of dimension $2g+1$.

We now define an invariant on quadratic forms which plays an important role in the classification of quadratic forms over $\FF_2$. 

\begin{definition}
Let $\{e_1,\ldots,e_g,f_1,\ldots,f_g\}$ be a symplectic basis of $(V,\langle,\rangle)$.
We define
the \emph{Arf invariant} $a(q)$ of a quadratic form $q$ by
$$a(q) = \sum_{i=1}^g q(e_i)q(f_i).$$
A quadratic form $q$ is called {\it odd} (resp. {\it even}) if $a(q)=1$ (resp. $a(q)=0$). Let $\QVO$ (resp. $\QVE$) denote the set of all odd (resp. even) quadratic forms. 
\end{definition}

The Arf invariant does not depend on the choice of symplectic basis. The invariant splits the quadratic forms into two classes $\QVO$ and $\QVE$ which have the cardinalities $2^{g-1}(2^g-1)$ and $2^{g-1}(2^g+1)$ respectively.

\subsubsection{Quadratic forms in terms of coordinates}\label{sec:quadcoor}

We may introduce the quadratic forms also in terms of coordinates by fixing a symplectic basis. 

Fix a symplectic basis $\{e_1,\dots , e_g;f_1,\dots ,f_g\}$. We write the linear expression of any vector $w\in V$ as follows
\begin{equation*}
w=\lambda_1e_1+\dots +\lambda_ge_g+\mu_1f_1+\dots +\mu_gf_g.
\end{equation*}
For the simplicity, we write $w=(\lambda,\mu)$, where $\lambda=(\lambda_1,\dots , \lambda_g)$ and $\mu=(\mu_1,\dots , \mu_g)$ in $\FF_2^{g}$. We define the simplest quadratic form $q_0$ as
\begin{equation}\label{q0}
q_0(w)=\lambda\cdot \mu,
\end{equation}  
where $\cdot$ denotes the usual scalar product of $g$-tuples. If we take any vector $v\in V$ with the coordinates $(\epsilon,\epsilon')=(\epsilon_1,\dots,\epsilon_g,\epsilon_1',\dots , \epsilon_g')$ then the quadratic form $q:=q_0+v$ acts on $V$ by
\begin{equation*}
q(w)=\epsilon\cdot \mu +\epsilon'\cdot \lambda +\lambda \cdot \mu.
\end{equation*}
Let us write $q=\Ch{\eps}{\eps'}$. We see that 
$$\eps=(q(e_1),\dots,q(e_g)), \quad \eps'=(q(f_1),\dots,q(f_g)),$$
and so the Arf invariant of the quadratic form $q$ in coordinates is given as
$$a(q)=\eps \cdot \eps'.$$
In terms of coordinates, we have 
\begin{equation*}
\begin{array}{l}
\Ch{\eps}{\eps'} + (\lambda,\mu)=\Ch{\eps+\mu}{\eps'+\lambda},\\
\Ch{\eps_1}{\eps'_1}+\Ch{\eps_2}{\eps'_2}+
\Ch{\eps_3}{\eps'_3}=\Ch{\eps_1+\eps_2+\eps_3}
{\eps'_1+\eps'_2+\eps'_3}.
\end{array}
\end{equation*}
 This implies that 
\begin{align}
\begin{split}\label{arfproperty1}
&a(q+v)=a(q)+q(v),\\ 
&a(q_1+q_2+q_3)=a(q_1)+a(q_2)+a(q_3)+\langle v_1,v_2 \rangle, 
\end{split}
\end{align}
where $v_1=q_1+q_2, v_2=q_1+q_3$ for any $q,q_1,q_2,q_3\in \QV$ and $v\in V$.

\subsubsection{Aronhold Basis}

Let $S=\{q_1,\dots , q_{2g+1}\}$ be a set of linearly independent vectors of the vector space $V\bigsqcup \QV$, where all the vectors lie in $\QV$. Then any vector $q\in V\bigsqcup \QV$ can be written as the sum $\sum \alpha_i q_i$ with $\alpha_i=0,1 \in \ZZ$. We define the {\it length } of $q$ as the sum $\sum \alpha_i$. Denote it $\# q$.  
So we have $0\leq \# q \leq 2g+1.$ We remark that if $q$ is in the coset $\QV$ then $\#q$ is odd, since the sum of two quadratic forms corresponds to a unique vector in $V$.  

\begin{definition}
The set $S=\{q_1,\dots , q_{2g+1}\}$ is called an {\it Aronhold basis} if the Arf invariant of any element $q$ only depends on $\# q$ modulo $4$.
\end{definition}

An Aronhold basis exists \cite[Proposition 2.1]{GroHar2018}. Now, we introduce fundamental sets of $V$ which are closely related with the Aronhold bases, and see how to obtain an Aronhold basis from a fundamental set. 

 \begin{definition}\label{fundset} A set $\{v_1,\dots,v_{2g+1}\}$ of vectors in $V$ is called a {\it fundamental set} of $V$ if 
 \begin{itemize}
 \item $\sum\limits_{i=1}^{2g+1}v_i=0$ ({\it completeness}),
 \item $\langle v_i , v_j\rangle=1 $ for $1\leq i\neq j \leq$ ({\it being azygetic}).
 \end{itemize} 
\end{definition}

It is possible to obtain a fundamental set by an Aronhold basis. Indeed, if $\{q_1,\dots ,q_{2g+1}\}$ is an Aronhold basis then $\left\{q_1+q_2,\dots ,q_1+q_{2g+1}, \sum\limits_{i=2}^{2g+1}q_i\right\}$ is a fundamental set of $V$.

Conversely, we can obtain an Aronhold basis from a fundamental set as follows. Suppose that the set 
\begin{align*}
\mathscr{F}:=\{v_1,\dots , v_{2g+1} \}
\end{align*}
 is a fundamental set.  Now, let $q$ be any quadratic form. For $\mu\in \{0,1\}$, consider the set 
\begin{align*} \mathscr{E}_{q,\mu} := \{v_i\in \mathscr{F}\mid q(v_i)=\mu\}.
\end{align*}
Fix any $\mu \in \{0,1\}$. We may assume that $\mathscr{E}_{q,\mu}=\{v_1,\dots , v_k \}$ by reordering $\mathscr{F}$. We set $w:=\sum\limits_{i=1}^kv_i$ and $q_i=q+w+v_i$ for $i = 1,\dots , 2g+1$. Define $\mathcal{A}:=\{q_1,\dots , q_{2g+1}\}$. 

\begin{proposition}\label{prop:Aron} $\mathcal{A}$ is an Aronhold basis.
\end{proposition}
\begin{proof} First of all, we show that $\mathcal{A}$ spans $V\bigsqcup \QV$.  
Suppose that $v \in V$. Note that, the condition of being azygetic for a fundamental set implies that any subset of $\mathscr{F}$ with $2g$ elements forms a basis of $V$.
 Since any $2g$-subset of $\mathscr{F}$ forms a basis, we can write $v=v_{i_1}+\dots + v_{i_n}$ as a linear combination of vectors in $\mathscr{F}$. Thanks to the completeness property of $\mathscr{F}$, we may assume that $n$ is even. So $v=q_{i_1}+\dots +q_{i_n}$, since $\Char \mathbb{F}_2 = 2$.   
Now suppose that $q'$ is a quadratic form, then $q+q'$ is a vector $v\in V$. We write $v+w=v_{i_1}+\dots +v_{i_n}$ as a linear combination of $v_1,\dots , v_{2g}$. We may assume that $n$ is odd because of the completeness property in Definition~\ref{fundset}. Now, we have

\begin{align*} 
q'=\sum\limits_{j=1}^nq+w+v_{i_j}=\sum\limits_{j=1}^nq_{i_j}&.\\
\end{align*}
Therefore it forms a basis for $V\bigsqcup \QV$ since $\dim V\bigsqcup \QV=2g+1$.

Now, we need to show that the Arf invariant of any quadratic form only depends on its length modulo $4$. We start to prove it by starting with the quadratic forms in $\mathcal{A}$. So we compute 
\begin{align*} a(q_i)&=a(q+w+v_i)\\
&=a(q)+q(w+v_i)\\
&=a(q)+q(w)+q(v_i)+\langle w, v_i \rangle\\
&= 
 \begin{cases}
       a(q)+q(w)+\mu+k-1 \text{ \ if \ } i\in\{1,\dots , k\},\\
      a(q)+q(w)+\mu+1+k \text{ \ otherwise}.\\
     \end{cases}
\end{align*}
Both cases are equal modulo $2$. So we have $a(q_1)=\dots = a(q_{2g+1})$. We show that $a(q'')=a(q')+1$ for any $q',q'' \in  \QV$ with $\#q'' = \#q'+2$.

Let $\#q'=n$. Write $q'=q_{i_1}+\dots   +q_{i_{n}}$ and $q''=q_{j_1}+\dots  +q_{j_{n+2}}$ in terms of quadratic forms in $\mathcal{A}$.
\begin{align*}
a(q')&=a(q_{i_1}+\dots +  +q_{i_{n}})\\
&=a(q_{i_1})+a(q_{i_{2}})+a(q_{i_{3}}+\dots + q_{i_{n}}) + \langle q_{i_1}+q_{i_2},q_{i_1}+q_{i_{3}}+\dots + q_{i_{n}}\rangle \\ 
&=a(q_{i_1})+a(q_{i_{2}})+a(q_{i_{3}}+\dots + q_{i_{n}}) + \langle v_{i_1}+v_{i_2},v_{i_1}+v_{i_{3}}+\dots + v_{i_{n}}\rangle \\ 
&=a(q_{i_1})+a(q_{i_{2}})+\dots + a(q_{i_{n}})\\
&\quad+\langle v_{i_1}+v_{i_2},v_{i_1}+v_{i_{3}}+\dots + v_{i_{n}} \rangle +\langle v_{i_3}+v_{i_4},v_{i_3}+v_{i_{5}}+\dots + v_{i_{n}}\rangle \\
&\quad+\dots +\langle v_{i_{n-2}}+v_{i_{n-1}},v_{i_{n-2}}+v_{i_{n}} \rangle \\
&=a(q_{i_1})+a(q_{i_{2}})+\dots + a(q_{i_{n}})+(n-1)+(n-2)+\dots +2+1 \Mod 2\\
&=a(q_{i_1})+a(q_{i_{2}})+\dots + a(q_{i_{n}})+\frac{(n-1)}{2} \Mod 2.
\end{align*}
Similarly,
\begin{align*}
a(q'')=a(q_{j_1}+\dots +  +q_{j_{n+2}})=a(q_{j_1})+a(q_{j_{2}})+\dots + a(q_{j_{n+2}})+\frac{(n+1)}{2} \pmod2.
\end{align*}
Since all the quadratic forms in $\mathcal{A}$ have the same Arf invariant, we have $a(q'')=a(q')+1$. 
\end{proof}  

\begin{remark}
Proposition~\ref{prop:Aron} is a coordinate free reformulation of \cite[Theorem A.II.1.1]{RauFar1974}.
\end{remark}

The Aronhold bases enable us to determine the Arf invariant of any quadratic form. 

\begin{proposition}\label{prop:Aronevod} Let $\mathcal{A}$ be any Aronhold basis. For any $q\in \mathcal{A}$, we have 
\begin{align*}
a(q)=
 \begin{cases}
      0 \text{ \ \ \ for \ } g= 0,1\Mod 4,\\
      1 \text{ \ \ \ for \ } g= 2,3\Mod 4.
     \end{cases}
\end{align*}

\end{proposition} 
\begin{proof} Any quadratic form $q$ can be written uniquely as a linear combination of quadratic forms in $\mathcal{A}$, and $a(q)$ depends only on $\#q$ modulo 4. So if we count the lengths of the quadratic forms which are 1 modulo 4 as follows,
\begin{align*}
\sum\limits_{i=1 \Mod 4}^{2g+1}{{2g+1}\choose{i}}= 
 \begin{cases}
       2^{g-1}(2^g+1) \text{ \ \ \ for \ } g= 0,1\Mod 4,\\
      2^{g-1}(2^g-1) \text{ \ \ \ for \ } g= 2,3\Mod 4\\
     \end{cases}
\end{align*}
then the proposition follows since we have $2^{g-1}(2^g-1)$ and $2^{g-1}(2^g+1)$ odd and even quadratic forms respectively. 
\end{proof}

\begin{remark}\label{determination}
Now, we can determine the Arf invariant of any quadratic form from its length, since we know the Arf invariant of a quadratic form of length $1$. 
\end{remark}

\subsubsection{Labeling}\label{labelling}
%Aronhold bases lead us to study the quadratic forms with the subsets of the cardinality of  of cardinality odd as follows. Let $\calA=
We consider an Aronhold basis $\{q_1,\dots , q_{2g+1}\}$. Any quadratic form can be written uniquely as the sum of odd many $q_i$'s. 
\begin{align*}
&q_i \text{  \quad \quad \quad \quad \quad of  length } 1,\\
&q_i+q_j+q_k \text{ \quad of length } 3,\\
&\vdots\\
&\sum\limits_{i=1}^{2g+1}q_i \text{ \quad of length } {2g+1}.
\end{align*}

Thanks to Proposition~\ref{prop:Aronevod}, we can determine whether a quadratic form is even or odd from its length. 

In addition, we can label any quadratic form by an odd cardinality subset of $\{1,\dots,2g+1\}$ of odd cardinality. For an odd number $k$ in $\{1,\dots,2g+1\}$, the set $I:=\{i_1,\dots,i_k\}$ labels the quadratic form $q=\sum\limits_{j=1}^{k}q_{i_j}$ uniquely since the linear expression is unique. We call such a set $I$ the \emph{label}. We denote $q_I$ the quadratic form $q$. For our purpose, we are interested in labeling the quadratic forms. But, incidentally, note that any vector in $V$ can be labeled via even cardinality subsets of  $\{1,\dots,2g+1\}$ in the same way as the quadratic forms are labeled. 

Let $I_1,\dots, I_k$ be the labels for some quadratic forms on $V$ with $k$ is odd. Since $\QV \bigsqcup V$ is a vector space over $\FF_2$, the pairs of the same quadratic forms and the vectors are cancelled in the sum $q_{I_1}+\dots +q_{I_k}$. So it is labeled by $I_1\triangle \dots \triangle I_k$, where $\triangle$ denotes the symmetric difference of set. 

Finally, notice that once we fix an Aronhold basis the labeling is naturally unique. 

\subsubsection{Syzygetic Tetrads and Steiner Sets}\label{secsteset}

\begin{definition}\label{DefSyzAzy}
A set of three elements $q_1,q_2,q_3$ in $\QV$ is called a {\it syzygetic triad} (resp. {\it azygetic triad}) if 
\begin{equation*}
a(q_1)+a(q_2)+a(q_3)+a(q_1+q_2+q_3)=0 \text{\ (resp.\  $=1$)}. 
\end{equation*}
\end{definition}

A syzygetic triad $\{q_1,q_2,q_3\}$ can be completed into a set of four quadratic forms  
$$\{q_1,q_2,q_3,q_1+q_2+q_3\}$$ that adds up to zero. Such a set is called a {\it syzygetic tetrad}. By Definition~\ref{DefSyzAzy}, any $3$-subset of this tetrad forms a syzygetic triad. Being syzygetic reflects on the labeling of quadratic forms as follows. Suppose that $q_{I_i}$ is a quadratic form
labelled by $I_i\subset \{1,\dots , 2g+1\}$ for $1\leq i \leq 4$.
Notice that, $\{q_{I_i}\mid i=1,\dots ,4\}$ is a syzygetic tetrad if and only if $I_1\triangle \dots \triangle I_4=\emptyset$ since being syzygetic for $q_{I_i}$'s means that $\sum\limits_{i=1}^4q_{I_i}=0$.

Syzygetic tetrads yield some sets called {\it Steiner sets} which classify syzygetic tetrads of odd quadratic forms.

\begin{definition}
For any $v\in V$, we define the {\it Steiner set}
\begin{equation*}
{\bf S}_v:=\Big\{ q \in \QVO \mid  q+v\in \QVO\Big\}. 
\end{equation*}
\end{definition}

\begin{remark}\label{LabSyz} There are $2^{2g}-1$ Steiner sets. Each Steiner set has $2^{g-2}(2^{g-1}-1)$ elements paired by the translation $q\mapsto q+v$. Such two pairs form a syzygetic tetrad. 
\end{remark}

\begin{remark}\label{SteSet}
A characterization for an odd quadratic form $q$ to belong to ${\bf S}_v$ is the equality $q(v)=0$ holds. It follows from $a(q+v)=a(q)+q(v)$. 
Also note that
\begin{equation}\label{intSte}
\# {\bf S}_v\cap {\bf S}_{v'}=\begin{cases} 2^{g-1}(2^{g-2}-1) &
  \langle v,v'\rangle = 0, \\ 2^{g-2}(2^{g-1}-1) & \langle v,v' \rangle \neq 0. 
\end{cases}
\end{equation}

\end{remark}

\begin{corollary}\label{SymSte} Equation~\eqref{intSte} enables us to determine a symplectic basis so that a labeling from the set of Steiner sets as follows. Let $\mathfrak{S}$
be the set of all the Steiner sets. Now, we can construct a subset $\mathfrak{S}'=\{S_{1},\dots ,S_g, S'_{1},\dots , S'_g\}$ of $\mathfrak{S}$ with cardinality $2g$ such that 
\begin{align*}
& \#S_i\cap S_j = 2^{g-1}(2^{g-2}-1), \\
& \#S'_i\cap S'_j = 2^{g-1}(2^{g-2}-1), \\ 
& \#S_i\cap S'_j = \begin{cases}2^{g-2}(2^{g-1}-1) & \text{if } i=j, \\
2^{g-1}(2^{g-2}-1) & \text{if } i\neq j.
\end{cases} 
\end{align*} 
Note that the vectors corresponding to each element of $\mathfrak{S}'$ form a symplectic basis of $V$.  Call them $e_1,\dots ,e_g, f_1,\dots, f_g$ respectively.  Now, we can write the coordinates of a given quadratic form $q$ on $V$ with respect to $e_1,\dots ,e_g, f_1,\dots, f_g$ thanks to Remark~\ref{SteSet}. More precisely, we can compute the coordinates $q(e_i), q(f_i)$ by checking whether $q$ is contained in $S_i, S'_i$ or not for $i=1,\dots g$.  This procedure gives a labeling which is defined in Section~\ref{labelling} in terms of coordinates. 
\end{corollary}

\subsection{Theta functions and characteristics}
In this part, we review basic definitions and properties of theta functions, and we see how to relate them with the quadratic forms. 

Let $g\geq 0$ and 
\begin{equation*}
\mathbb{H}_{g}=\{\tau \in
  \textrm{GL}_{g}(\mathbb{C}) \mid  \; ^{t}{\tau} = \tau, \ \textrm{Im} \tau > 0\}
\end{equation*}
be the Siegel upper half space consisting of  $g\times g$ complex matrices with positive definite imaginary part. 

\begin{definition}\label{Def:theFun}
For $\tau \in \mathbb{H}_g$, $z=(z_1,\ldots,z_g) \in \mathbb{C}^g$ and 
$$
[q]  = \Ch{\varepsilon}{\varepsilon'} \in \mathbb{Z}^g \oplus \mathbb{Z}^g, 
$$
{\it the theta function with characteristic $[q]$} is 
$$
\vartheta[q](z,\tau) = \sum_{n \in \ZZ^g}
\exp\left(\pi i (n + \varepsilon/2) \tau {^t (n + \varepsilon/2)} + 
2\pi i(n + \varepsilon/2){^t (z+\varepsilon'/2)}\right).
$$
\end{definition}

This is an analytic function on $ \CC^g \times \HH_{g}$. The evaluation of $\vartheta[q](z,\tau)$ at $z=0$ is called a {\it theta constant} ({\it Thetanullwert}) (with characteristic $[q]$), which is denoted by $\vartheta[q](\tau)$. The characteristic $[q]$ is called {\it even} (resp. {\it odd}) if $\varepsilon\cdot \varepsilon'$ is even (resp. odd). Since
\begin{equation}\label{eq:evenoddchar}
\vartheta \Ch{\eps}{\eps' }(-z,\tau) = (-1)^{\eps \cdot \eps'}
\cdot  \vartheta \Ch{\eps}{\eps'}(z,\tau),
\end{equation}
\cite[Theorem I.2]{RauFar1974}, the theta function $\vartheta \Ch{\eps}{\eps' }$ is even (resp. odd) if and only if $\varepsilon\cdot \varepsilon'=0$ (resp. $=1$).
In addition, note that a characteristic $[q]$ is odd if and only if the theta constant $\vartheta[q](\tau)$ is identically $0$ for all $\tau \in \HH_g$. We also have that \cite[Theorem I.3]{RauFar1974}
\begin{equation}\label{eq:charmodtwo}
\vartheta \Ch{\eps + 2m}{\eps' + 2 n }(z,\tau) = (-1)^{n \cdot \eps}
\cdot  \vartheta \Ch{\eps}{\eps'}(z,\tau).
\end{equation}

Using the notation of Section~\ref{sec:quadcoor}, we identify a characteristic $[q]$ modulo $2$ with a quadratic form over $\FF_2$ which is denoted by $q$. The quadratic from $q_0$ defined in ~\eqref{q0} is identified with the characteristic $\Ch{0}{0}$. Conversely, fixing a symplectic basis, if we start with a quadratic form $q$ then we write $q=\Ch{\eps}{\eps'}$ with entries $\{0,1\}$ in terms of coordinates. We associate $\Ch{\eps}{\eps'}$ to the characteristic of the theta function $\vartheta \Ch{\eps + 2m}{\eps' + 2 n }(z,\tau)$ for all $n,m \in \ZZ$. The characteristic $\Ch{\eps}{\eps'}$ has only an impact on the sign of the theta function because of Equation~\eqref{eq:charmodtwo}. From now on, we only use characteristics with entries $0,1$. 

Now, we specify a particular symplectic vector space over $\FF_2$ related to algebraic curves. From now on, unless otherwise stated, we let $\calC$ to be a non-hyperelliptic curve of genus $g>0$
$\C$ and $\boldsymbol{\omega}=(\omega_1,\ldots,\omega_g)$ be a basis of regular
differentials on $\calC$. Let $\boldsymbol{\delta}=(\delta_1,\ldots,\delta_{2g})$ be a symplectic basis of
$H_1(\calC,\Z)$ such that the intersection pairing has the matrix
\begin{align*}
 \begin{pmatrix} 0_g & I_g \\
I_g & 0_g 
\end{pmatrix}
\end{align*}
with $I_g$ and $0_g$ are the $g\times g$ identity and zero matrices respectively.

With respect to these choices, the period matrix of $\calC$ is
$\Omega=[\Omega_1,\Omega_2]$, where $$\Omega_1=\left(\int_{\delta_i} \omega_j\right)_{1
\leq i \leq g, 1
\leq j \leq g},$$ 
$$\Omega_2=\left(\int_{\delta_i} \omega_j\right)_{ g+1 \leq i \leq 2g,1 \leq j
  \leq g}.$$ We consider a second basis $\boldsymbol{\eta}=(\eta_1,\dots , \eta_g)$  of
regular differentials  obtained by $\boldsymbol{\eta}=
\Omega_1^{-1} \boldsymbol{\omega}$. The period matrix with respect to this new
basis is $[\textrm{id}, \tau]$, where $\tau=\Omega_1^{-1} \Omega_2 \in
\mathbb{H}_g$. This matrix is called the {\it Riemann matrix}. We let 
$$\Jac(\calC) = \CC^g/(\ZZ^g + \tau \ZZ^g).$$

Let us denote 
$$e_i = {\left(\frac{1}{2} \int_{\delta_i} \eta_j \right)_{1\leq j
    \leq g}} =  (0,\ldots,0,\frac{1}{2},0,\ldots,0) \in \CC^g$$
and
$$f_i = { \left(\frac{1}{2} \int_{\delta_{g+i}} \eta_j\right)_{1\leq j
    \leq g}} \in \CC^g,$$
and
$$v=\sum_{i=1}^g \lambda_i e_i + \mu_j f_j = (\lambda,\mu)$$
with $\lambda,\mu \in\ZZ^g$ for $1 \leq i \leq g$.

Now, we let $W$ be the $\ZZ$-module generated
by $e_1,\ldots,e_g,f_1,\ldots,f_g$ so that 
\begin{align*}
\Jac(\calC)[2] = W/(\ZZ^g+\tau \ZZ^g).
\end{align*}
An element $v\in W$ acts on a theta function. Indeed, if $[q]=\Ch{\epsilon}{\epsilon'}$ is a characteristic and  
$v=(\lambda,\mu) \in W$ then 
\begin{equation}\label{lem:actonchar} 
\vartheta[q](z+v,\tau)=\ex\left(-\frac{1}{4} \mu {^t (\epsilon'+\lambda)} -
  \frac{1}{2} \mu
{^t z} - \frac{1}{8} \mu \tau {^t\mu}\right) \cdot \vartheta \car{\epsilon +
  \mu}{ \epsilon' + \lambda}(z,\tau)
\end{equation} \cite[Theorem I.5]{RauFar1974}.

Thanks to Equation~\eqref{lem:actonchar}, we write $\big[[q]+v\big]=\car{\epsilon + \mu}{ \epsilon' + \lambda}$. 

At the same time, $V=\Jac(\calC)[2]$ is a vector space over $\FF_2$ of $2g$ dimension. The Weil pairing defines a nondegenerate symplectic form on $V$. We may induce the symplectic basis of $V$ via $e_i, f_i$'s. Now thanks to identifications in Section~\ref{sec:quadcoor}, the theory of quadratic forms on $V$ is coherent with the theta characteristics and $(\lambda,\mu)$ modulo $2$. We denote $\bar{v} \in V$ the class of $v$, where the class is identified with the vector $(\lambda \Mod 2,\mu \Mod 2)$. So the quadratic form $q+\bar{v}$ is the quadratic form associated to the theta characteristic $\big[[q]+v\big]$.

\subsection{Theta characteristic divisors}\label{thetaCharDiv}

In this section, we introduce theta characteristic divisors of $\calC$. Moreover, we explain the link between such divisors and the quadratic forms on $\Jac(\calC)[2]$ over $\FF_2$. For more results and detailed explanations, we refer to \cite[Chapter 1]{ArbCorGriHar1985}.

Let $\calC_d$ be the $d$-fold symmetric product of $\calC$ which is identified with the set of effective divisors of degree $d$. Fix a point $Q$ on $\calC$. The {\it Abel-Jacobi} map is defined by 
\begin{align*}
u_d: \calC_d &\longrightarrow \Jac(\calC) \\ 
D=\sum\limits_{i}m_iP_i &\longmapsto \sum\limits_{i}m_i\int_Q^{P_i}(\eta_1 , \dots , \eta_g).
\end{align*}

The map depends on the choice of the fixed point $Q$. Also, the value of the integral depends on the path chosen to integrate, however, $u_d(D)$ is well defined in $\Jac(\calC)$. It is possible to extend $u_d$ to noneffective divisors of degree $d$. Abel's theorem~\cite[Chapter 1]{ArbCorGriHar1985} assures that this map is invariant under the linear equivalence between divisors. Denote $\Pic(\calC)$ the Picard group of $\calC$ and $\Pic^d(\calC)$ the subgroup of the divisor classes of degree $d$ in $\Pic(\calC)$. By Abel's theorem, $u_d$ leads to a bijection from $\Pic^d(\calC)$ into $\Jac(\calC)$. Moreover, it induces an isomorphism between the group $\Pic^0(\calC)$ of the divisor classes of degree $0$ in $\Pic(\calC)$ and the Jacobian $\Jac(\calC)$. We keep these identifications in mind while we are studying theta characteristic divisors in the following part. 

The {\it Riemann Theta function} $\theta(z,\tau)$ of $\Jac(\calC)$ is the theta function $\vartheta\Ch{0}{0}(z,\tau)$ with the characteristic $\Ch{0}{
0}$. 

Since it is an analytic function on $\CC^g \times\HH_{g} $ and quasi periodic with respect to the lattice $\ZZ^g + \tau \ZZ^g$ given by $[\textrm{id}, \tau]$
 it defines a divisor $\Theta $ of $\Jac(\calC)$ which is the zero divisor of $\vartheta(z,\tau)$.

We denote $\ell(D)$ the dimension of the Riemann-Roch space of $D$. The following theorem allows us to relate certain divisors with quadratic forms. 

\begin{theorem}[Riemann Singularity Theorem]\label{riesinthe}
Let $\kappa$ be the
canonical divisor of $\calC$. 
There exists a unique divisor class $D_0$ of degree $g-1$ with $2 D_0 \sim
\kappa$ and $\ell(D_0)$ is even
such that
$u_{g-1}(\calC_{g-1}) = \Theta + u_{g-1}(D_0).$
Moreover for any $v \in V$, $\Mult_{v}(\Theta) = \ell(D_0+v)$.
\end{theorem}

A divisor (class) $D$ is called a {\it theta characteristic divisor (class)} if $2D \sim \kappa$. 
Now we have a correspondence between theta characteristic divisors of $\calC$ and quadratic forms on $\Jac(\calC)[2]$ over $\FF_2$ as follows.

Now, define $q_{D_0}(v):=\ell(D_0+v) \Mod{2}$. If $v$ is given by $(\lambda,\mu)$ with respect to a fixed symplectic basis then by Theorem~\ref{riesinthe} and Equation~\eqref{lem:actonchar} we have
\begin{align}\label{q0again}
q_{D_0}(v)&=\ell(D_0+v) \Mod{2} = \Mult_{v}(\Theta) \Mod{2} = \Mult_{v} \Bigg( \vartheta\Ch{0}{0} \Bigg) \Mod{2} \\
&= \Mult_{0} \Bigg( \vartheta\Ch{\lambda}{\mu} \Bigg) \Mod{2} = \lambda \cdot \mu \nonumber.
\end{align} 
We identify $q_{D_0}$ with $q_0$ that is defined in Equation~\eqref{q0}. 
Furthermore, any theta characteristic divisor $D$ is linearly equivalent to $D_0+v$ with $v=(\lambda , \mu)\in V$. Indeed, $D-D_0$ is a $2$-torsion point of $\Jac(\calC)$. We can associate $D$ to the quadratic form $q_D:=q_0+v$. Note that the Arf invariant of $q_D$
$$a(q_D)=a(q_0+v)= \Mult_{v}(\Theta) \Mod 2$$ since $\Mult_{v}(\Theta)$ is equal to the multiplicity of $\vartheta [q_D](z,\tau)$ at $0$ and the latter has the same parity as $q_D$. Thanks to Equation~\eqref{arfproperty1} and Theorem~\ref{riesinthe}, we have 
$$q_D(w)=a(q_D+w)+a(q) = \ell(D+w) + \ell(D) \Mod{2}.$$
for any $w\in V$.

Conversely, any quadratic form $q$ defines a divisor 
\begin{align}\label{corrthediv}
D_q:=D_0+q_0+q.
\end{align}

\begin{remark} \label{TheNoOfOddAndEvenTheDiv} To sum up, there is a one-to-one correspondence between the set of quadratic forms on $\Jac(\calC)[2]$ over $\mathbb{F}_2$ and the set of theta characteristic divisors. Moreover, an odd (resp. even) theta characteristic corresponds to an odd (resp. even) quadratic form. In the following section, we discuss extrinsic geometric objects which correspond to the theta characteristic divisors. 

\end{remark}

\subsection{Multitangents}

The basis of regular differentials $\{ \omega_1, \dots, \omega_g\}$ defines the canonical map 
\begin{align*}
\Phi : \calC &\rightarrow \PP^{g-1} \\
P & \mapsto (\omega_1(P):\dots : \omega_g(P)).
\end{align*}

Let $D$ be an effective theta characteristic divisor of $\calC$. We call $D$ a \emph{vanishing theta characteristic divisor} if $\ell(
D) > 1$. Note that $\Phi^*(\mathcal{O}_{\PP^{g-1}}(1))=\kappa$. Now, we let $H_D$ be any fixed hyperplane in $\PP^{g-1}$ such that $\Phi^*H_D\cdot \calC \sim 2D$.

\begin{definition}
We call such a hyperplane a {\it multitangent}. 
\end{definition}

\begin{remark} When $g=3$, the dimension of the Riemann-Roch space of any theta characteristic divisor is either $0$ or $1$ because $\calC$ is non-hyperelliptic. There are $28$ multitangents, in this case these geometric objects are known as ${\it bitangents}$.

For $g=4$, first of all, for any theta characteristic divisor $D$, $\ell(D)=0,1 \text{ or } 2 $ because of Clifford's theorem for divisors~\cite[Chapter III]{ArbCorGriHar1985}. The multitangents are known as {\it tritangents} in this case. There are $120$ tritangents which correspond to the effective odd characteristic divisors. To be more specific, the canonical model of $\calC$ lies on a smooth quadric if and only if there is not a vanishing theta characteristic divisor. In this case, we have exactly 120 tritangents. Otherwise, $C$ lies on a singular quadric $\mathcal{Q}$, then there is a unique effective even theta characteristic divisor, call $D_{\text{e}}$. The dimension $\ell(D_{\text{e}})=2$. So there is a one dimensional family of tritangents, that pass through the node of $\mathcal{Q}$. The tritangents which correspond to the effective odd theta characteristic divisors are the ones which do not pass through the node of $\mathcal{Q}$. Note that, such a curve arises from a del Pezzo surface of degree $1$ which follows from \cite[Theorem 24.4.iii]{Man1974}. In Section~\ref{DelPezzoSubsections}, we come back to this subject. 
\end{remark}

\begin{remark} Everything aside, if $\calC$ is a general curve of genus $g$, then $\ell(D)=0,1$ for any theta characteristic divisor $D$. So there is a unique hyperplane $H_D$ if $D$ is an effective odd theta characteristic divisor. In this case, notice that we have exactly $2^{g-1}(2^g-1)$ multitangents. For the generality condition, we refer to \cite{Har1982}. 
\end{remark}

\section{Computation of Theta Constants}\label{ALG}

In this section, we apply results from Section~\ref{background} to prove Theorem~\ref{mainTheorem}. This enables us to obtain Algorithm~\ref{AlgTheCon}. 
Recall that $\calC$ is a non-hyperelliptic curve of genus g.
Throughout the section, we fix a Riemann matrix $\tau$ of $\calC$ associated a normalized regular differentials $\boldsymbol{\eta}$ as introduced in Section~\ref{thetaCharDiv}. Thus we do not write $\tau$ in the notation of theta functions and constants. Recall that we denote $D_q$ the corresponding effective theta characteristic divisor to the quadratic form $q$. In addition, fix a theta hyperplane $H_{D_q}$ and also a linear polynomial $\beta_{{q}} \in \CC[X_1,\ldots,X_{g}]$ such that $H_{D_{q}}$ is the hyperplane with equation $\beta_{{q}} =0$.  We abuse the notation by identifying the canonical model of $\calC$ with itself. In addition, we expect the reader to be aware all the identifications which are described in Section~\ref{background} among quadratic forms, theta characteristics, theta characteristic divisors and multitangents.

We begin with two even characteristics $p_1,p_2$ and write $p_1+p_2=q_1+\overline{q}_1$. Let  $D_{q_1}, D_{\overline{q}_1}$ be the theta characteristics divisors associated to $q_1, \overline{q}_1$. We write
\begin{align*}
D_{q_1}\sim A_1+\dots +A_{g-1}, \text{  } D_{\overline{q}_1}\sim B_1+\dots +B_{g-1},
\end{align*}
where $A_i$'s (respectively $B_i$'s) are the tangency points  of the multitangent $\beta_{q_1}$ (respectively $\beta_{\overline{q}_1}$) for $i = 1,\dots , g-1$. 
Let $T=T_1+\dots + T_{2g-3}$ be an arbitrary generic
effective divisor of degree $2g-3$ on $\calC$ and be $\kappa=2(A_1+\dots + A_{g-1})$. By fixing a point $P_0$ on $\calC$, we introduce
$$f_{i,T}(P) := \vartheta[p_i](P+T-\kappa):=\vartheta[p_i]\left(\int_{P_0}^P\boldsymbol{\eta}+\sum_{i=1}^{2g-3}\int_{P_0}^{T_i}\boldsymbol{\eta}-2\sum_{i=1}^{g-1}\int_{P_0}^{A_i}\boldsymbol{\eta}\right).$$

According to Riemann theorem \cite[Theorem V.1]{RauFar1974}, $f_{i,T}(P)$ is a regular section of a line bundle
over $\calC$, and if
$f_{i,T}$ is not identically zero then its zero divisor $(f_{i,T})_0$
has degree $g$ and satisfies
$$(f_{i,T})_0 \sim D_0+ (p_i + q_0) + \kappa-T = D_{p_i}
+\kappa -T.$$
Since $\ell(\kappa + D_{p_i})=2g-2$, we let $\{t^{(1)}_i,\dots,t^{(2g-2)}_i\}$ be a basis of sections on the line bundle $[\kappa+D_{p_i}]$ (called \emph{Wurzelfunctionen} in Weber's book) which corresponds to a basis of $\calL({\kappa + D_{p_i}})$. Suppose that $t_i^{(j)}$ is given by the family of rational functions $t_{i,\alpha}^{(j)}$ with an open cover $\{U_{i,\alpha}\}_{\alpha\in I}$ of $\calC$ for $i=1,2$ and $j=1,\dots , 2g-2$. 

For each $i=1,2$, we can find an open cover $\{U_{i,\alpha}\}_{{\alpha}}$ of $\calC$ such that for each $k=1,\dots , 2g-3$ there exists $\alpha_k $ for which $T_k$ is not a pole of $t^{(j)}_{i,\alpha_k}$ for any $j=1,\dots ,2g-3$.
We define $\chi_{i,T}$ as the family of the following rational functions
\begin{equation}\label{det}
\chi_{i,T,\alpha}(P) =  \begin{vmatrix} t^{(1)}_{i,\alpha}(P) & \cdots &
  t^{(2g-2)}_{i,\alpha}(P) \\
t^{(1)}_{i,\alpha_{1}}(T_1) & \cdots &
  t^{(2g-2)}_{i,\alpha_1}(T_1) \\
\vdots & & \vdots \\
t^{(1)}_{i,\alpha_{2g-3}}(T_{2g-3}) & \cdots &
  t^{(2g-2)}_{i,\alpha_{2g-3}}(T_{2g-3}) \\
\end{vmatrix}, \quad 1 \leq i \leq 2,
\end{equation}
on $U_\alpha$ for all $\alpha$. Therefore, the sections $\chi_{i,T}$ and $t_{i}^{(j)}$ are of the same line bundle $[\kappa+D_{p_i}]$, since the determinant is a linear combination of $t_{i,\alpha}^{(j)}$'s. 

Since $\chi_{i,T}(T_j)=0$ for $1
\leq j \leq 2g-3$, we see that $(\chi_{1,T})_0 = T + R_i$, where $R_i$
is an effective divisor
of degree $g$, uniquely defined by $R_i+T \sim \kappa + D_{p_i}$.
Now $$(f_{i,T})_0 \sim D_{p_i}
+\kappa -T \sim R_i,$$
so actually $(f_{i,T})_0=R_i$. Therefore, $(f_{1,T})_0 - (f_{2,T})_0 =
R_1-R_2= (\chi_{1,T})_0 -
(\chi_{2,T})_0$ and there exists a constant $\lambda_T$ such that 
$$\frac{f_{1,T}(P)}{f_{2,T}(P)} = \lambda_T \cdot \frac{\chi_{1,T}(P)}{\chi_{2,T}(P)}.$$  

\begin{lemma} $\lambda_T$ does not depend on $T$.
If $T = A_2+\dots +A_{g-1}+A_1+\dots +A_{g-1}$ then
$$\frac{f_{1,T}(A_1)^2}{f_{2,T}(A_1)^2} = \frac{\vartheta[p_1](0)^2}{\vartheta[p_2](0)^2}.$$
If moreover $T'=A_2+\dots +A_{g-1}+ B_1+\dots +B_{g-1}$ then 
$$\frac{f_{1,T'}(P)^2}{f_{2,T'}(P)^2} = (-1)^{a(q_0+p_1+p_2)}
\cdot \frac{f_{2,T}(P)^2}{f_{1,T}(P)^2}.$$
\end{lemma}

\begin{proof} The proof is a direct generalization of \cite[Lemma 3.1, Lemma 3.2]{NarRit2017}.
\end{proof}

From this we get that
$$\frac{f_{1,T}(A_1)^2 \cdot f_{2,T'}(A_1)^2}{f_{2,T}(A_1)^2 \cdot
  f_{1,T'}(A_1)^2} = (-1)^{a(q_0+p_1+p_2)} \cdot  \frac{\vartheta[p_1](0)^4}{\vartheta[p_2](0)^4}
= \frac{\chi_{1,T}(A_1)^2 \cdot
  \chi_{2,T'}(A_1)^2}{\chi_{2,T}(A_1)^2 \cdot \chi_{1,T'}(A_1)^2}.$$

We denote
$\sqrt{q}$ a (fixed) section (\emph{Abelsche Function}\label{AbelscheFun}) of the line bundle associate to $D_{q}$ for a quadratic form $q$. We write $\sqrt[P]{q}$ instead of $\sqrt{q}(P)$ to express all the following matrices relatively more decent. % as Weber did. 
Let 
\begin{align*}
\Big\{ \{r_i,\overline{r}_i\} \mid i=1,\dots, g-1 \Big\} \text{ and } \Big\{ \{s_i,\overline{s}_i\} \mid i=1,\dots, g-1\Big\}
\end{align*}
be the sets of $g-1$ many distinct pairs such that 
\begin{align}\label{choiceSte1}
\begin{split}
p_1+q_1 = r_i+\overline{r}_i \text{ for } i=1,\dots , g-1, \\  
p_1+\overline{q}_1 = s_i+\overline{s}_i \text{ for } i=1,\dots , g-1.
\end{split}
\end{align}
\begin{remark}\label{enoughChoice}
Notice that $r_i, \overline{r}_i\in{\bf S}_{p_1+q_1}$ and $s_i, \overline{s}_i\in{\bf S}_{p_1+\overline{q}_1}$ for $i=1,\dots , g-1$. 
For each align in Equation~\eqref{choiceSte1}, we have $2^{g-2}(2^{g-1}-1)$ pairs of odd quadratic forms satisfying them, see Remark \ref{LabSyz}. 
\end{remark}

Assume that we can set the following expressions
\begin{align}\label{choiceBasis}
\begin{split}
t^{(j)}_1&=\sqrt{q_1r_j\overline{r}_j} \text{ for } j\in\{1,\dots, g-1\}, \text{ } t^{(j)}_1=\sqrt{\overline{q}_1s_j\overline{s}_j} \text{ for } j\in\{g,\dots, 2g-2\}, \\
t^{(j)}_2&=\sqrt{q_1s_j\overline{s}_j} \text{ for } j\in\{1,\dots, g-1\}, \text{ } t^{(j)}_2=\sqrt{\overline{q}_1r_j\overline{r}_j} \text{ for } j\in\{g,\dots, 2g-2\}. 
\end{split}
\end{align}

Once we make the choice in Equation~\eqref{choiceBasis}, the quotient 
$\chi_{1,T}(A_1)/\chi_{2,T}(A_1)$ take the indeterminate form $0/0$, so we need firstly to resolve this
ambiguity. 

\subsection{Resolving The Ambiguity}
We reset $T=T_2+\dots +T_{g-1}+A_1+\dots+A_{g-1}$. Note that $\sqrt[A_i]{q_1r_j\overline{r}_j}=0$ and $\sqrt[A_i]{q_1s_j\overline{s}_j}=0$, since $A_i$ is in the zeroes of the divisors $\Div\big({\sqrt{q_1r_j\overline{r}_j}}\ \big)$ and $\Div\big({\sqrt{q_1s_j\overline{s}_j}}\ \big)$ for $i=1,\dots, g-1$ and $j=1,\dots, g-1$. By using these identities, we have 
\begin{align*}
\chi_{1,T}(P) =g_{T}(P)
\begin{vmatrix} 
\sqrt[P]{r_1\overline{r}_1}  & \cdots & \sqrt[P]{r_{g-1}\overline{r}_{g-1}} \\
\sqrt[T_2]{ r_1\overline{r}_1}  & \cdots & \sqrt[T_2]{r_{g-1}\overline{r}_{g-1}} \\
\vdots & & \vdots \\
\sqrt[T_{g-1}]{r_1\overline{r}_1}  & \cdots & \sqrt[T_{g-1}]{ r_{g-1}\overline{r}_{g-1}}  
\end{vmatrix} 
 \begin{vmatrix} \sqrt[A_1]{ s_1\overline{s}_1}  & \cdots & \sqrt[A_1]{s_{g-1}\overline{s}_{g-1}} \\ 
\vdots & & \vdots  \\ \\
 \sqrt[A_{g-1}]{ s_1\overline{s}_1}  & \cdots & \sqrt[A_{g-1}]{s_{g-1}\overline{s}_{g-1}} \\
\end{vmatrix},
\end{align*} 
and 
\begin{align*}
\chi_{2,T}(P)=g_{T}(P)
 \begin{vmatrix} \sqrt[P]{ s_1\overline{s}_1}  & \cdots & \sqrt[P]{ s_{g-1}\overline{s}_{g-1}} \\
\sqrt[T_2]{s_1\overline{s}_1}  & \cdots & \sqrt[T_2]{ s_{g-1}\overline{s}_{g-1}} \\
\vdots & & \vdots \\
\sqrt[T_{g-1}]{ s_1\overline{s}_1}  & \cdots & \sqrt[T_{g-1}]{ s_{g-1}\overline{s}_{g-1}}  
\end{vmatrix}
 \begin{vmatrix} \sqrt[A_1]{ r_1\overline{r}_1}  & \cdots & \sqrt[A_1]{r_{g-1}\overline{r}_{g-1}} \\
\vdots & & \vdots \\ \\
\sqrt[A_{g-1}]{ r_1\overline{r}_1}  & \cdots & \sqrt[A_{g-1}]{ r_{g-1}\overline{r}_{g-1}} 
\end{vmatrix}, \\
\end{align*} 
where $g_{T}(P)=\sqrt[P]{q_1}\sqrt[T_2]{q_1}\cdots \sqrt[T_{g-1}]{q_1}\sqrt[A_1]{\overline{q}_1}\cdots\sqrt[A_{g-1}]{\overline{q}_1}$. Now, let $P=A_1$ and $T_i=A_i$ for $i= 2, \dots, g-1$. So we have $\frac{\chi_{1,T}(A_1)}{\chi_{2,T}(A_1)}=1$. 

Hence we have
\begin{align}\label{sadelesme}
(-1)^{a(q_0+p_1+p_2)} \cdot  \frac{\vartheta[p_1](0)^4}{\vartheta[p_2](0)^4}
= \frac{
  \chi_{2,T'}(A_1)^2}{ \chi_{1,T'}(A_1)^2}.
\end{align}

\subsection{Rewriting The Quotient}
In order to have $(-1)^{a(q_0+p_1+p_2)} \frac{\vartheta[p_1](0)^4}{\vartheta[p_2](0)^4}$, we need to compute $\frac{\chi_{2,T'}(A_1)^2}{ \chi_{1,T'}(A_1)^2}$ thanks to Equation~\eqref{sadelesme}. Notice that $\chi_{1,T'}^2, \chi_{2,T'}^2$ are sections of the same line bundle corresponding to $3\kappa$. So their quotient is a rational function on the curve. All the following computations in this section are carried out to find this rational function.

Now, we suppose that $T'=A_2+\dots +A_{g-1}+ B_1+\dots +B_{g-1}$. 

Therefore, we have

\begin{align*}
\chi_{1,T'}(A_1) =  \begin{vmatrix} \sqrt[A_1]{q_1r_1\overline{r}_1}  & \cdots & \sqrt[A_1]{q_1r_{g-1}\overline{r}_{g-1}} & \sqrt[A_1]{\overline{q}_1s_1\overline{s}_1}  & \cdots & \sqrt[A_1]{\overline{q}_1s_{g-1}\overline{s}_{g-1}} \\
\vdots & & \vdots \\
\sqrt[A_{g-1}]{q_1r_1\overline{r}_1}  & \cdots & \sqrt[A_{g-1}]{q_1r_{g-1}\overline{r}_{g-1}} & \sqrt[A_{g-1}]{\overline{q}_1s_1\overline{s}_1}  & \cdots & \sqrt[A_{g-1}]{\overline{q}_1s_{g-1}\overline{s}_{g-1}} \\
\sqrt[B_{1}]{q_1r_1\overline{r}_1}  & \cdots & \sqrt[B_1]{q_1r_{g-1}\overline{r}_{g-1}} & \sqrt[B_1]{\overline{q}_1s_1\overline{s}_1}  & \cdots & \sqrt[B_1]{\overline{q}_1s_{g-1}\overline{s}_{g-1}} \\
\vdots & & \vdots \\
\sqrt[B_{g-1}]{q_1r_1\overline{r}_1}  & \cdots & \sqrt[B_{g-1}]{q_1r_{g-1}\overline{r}_{g-1}} & \sqrt[B_{g-1}]{\overline{q}_1s_1\overline{s}_1}  & \cdots & \sqrt[B_{g-1}]{\overline{q}_1s_{g-1}\overline{s}_{g-1}} \\
\end{vmatrix}. 
\end{align*} \\

Note that $\sqrt[A_i]{q_1r_{j}\overline{r}_{j}}=0$, $\sqrt[A_i]{q_1s_{j}\overline{s}_{j}}=0$ and $\sqrt[B_i]{\overline{q}_1r_{j}\overline{r}_{j}}=0$, $\sqrt[B_i]{\overline{q}_1s_{j}\overline{s}_{j}}=0$ for $i=1,\dots, g-1$, $j=1,\dots , g-1$. Then we have

\begin{align*}
\chi_{1,T'}(A_1)& = \begin{vmatrix} 0  & \cdots & 0 & \sqrt[A_1]{\overline{q}_1s_1\overline{s}_1}  & \cdots & \sqrt[A_1]{\overline{q}_1s_{g-1}\overline{s}_{g-1}} \\
\vdots & & \vdots \\
0  & \cdots & 0  & \sqrt[A_{g-1}]{\overline{q}_1s_1\overline{s}_1}  & \cdots & \sqrt[A_{g-1}]{\overline{q}_1s_{g-1}\overline{s}_{g-1}} \\
\sqrt[B_{1}]{q_1r_1\overline{r}_1}  & \cdots & \sqrt[B_1]{q_1r_{g-1}\overline{r}_{g-1}} & 0  & \cdots & 0 \\
\vdots & & \vdots \\
\sqrt[B_{g-1}]{q_1r_1\overline{r}_1}  & \cdots & \sqrt[B_{g-1}]{q_1r_{g-1}\overline{r}_{g-1}} & 0  & \cdots & 0 \\
\end{vmatrix} \\ \\
&=c_{T'}
{\begin{vmatrix} 
\sqrt[B_1]{r_1\overline{r}_1}  & \cdots & \sqrt[B_1]{r_{g-1}\overline{r}_{g-1}} \\
\vdots & & \vdots \\
\sqrt[B_{g-1}]{r_1\overline{r}_1}  & \cdots & \sqrt[B_{g-1}]{ r_{g-1}\overline{r}_{g-1}}  
\end{vmatrix}}
 {\begin{vmatrix} \sqrt[A_1]{ s_1\overline{s}_1}  & \cdots & \sqrt[A_1]{s_{g-1}\overline{s}_{g-1}} \\ 
\vdots & & \vdots  \\ \
 \sqrt[A_{g-1}]{ s_1\overline{s}_1}  & \cdots & \sqrt[A_{g-1}]{s_{g-1}\overline{s}_{g-1}} 
\end{vmatrix}},
\end{align*} 
where  $c_{T'}=\sqrt[A_1]{\overline{q}_1}\cdots \sqrt[A_{g-1}]{\overline{q}_1}\sqrt[B_1]{{q}_1}\cdots\sqrt[A_{g-1}]{{q}_1}.$
After making similar computations for $\chi_{2,T'}$, we have then the following quotient

\begin{align*}
\frac{\chi_{2,T'}(A_1)}{\chi_{1,T'}(A_1)} 
&=\frac
{
\begin{vmatrix} 
\sqrt[B_1]{s_1\overline{s}_1}  & \cdots & \sqrt[B_1]{s_{g-1}\overline{s}_{g-1}} \\
\vdots & & \vdots \\
\sqrt[B_{g-1}]{s_1\overline{s}_1}  & \cdots & \sqrt[B_{g-1}]{ s_{g-1}\overline{s}_{g-1}}  
\end{vmatrix}
 \begin{vmatrix} \sqrt[A_1]{ r_1\overline{r}_1}  & \cdots & \sqrt[A_1]{r_{g-1}\overline{r}_{g-1}} \\ 
\vdots & & \vdots  \\ \
 \sqrt[A_{g-1}]{ r_1\overline{r}_1}  & \cdots & \sqrt[A_{g-1}]{r_{g-1}\overline{r}_{g-1}}
\end{vmatrix}}
{\begin{vmatrix} 
\sqrt[B_1]{r_1\overline{r}_1}  & \cdots & \sqrt[B_1]{r_{g-1}\overline{r}_{g-1}} \\
\vdots & & \vdots \\
\sqrt[B_{g-1}]{r_1\overline{r}_1}  & \cdots & \sqrt[B_{g-1}]{ r_{g-1}\overline{r}_{g-1}}  
\end{vmatrix}
 \begin{vmatrix} \sqrt[A_1]{ s_1\overline{s}_1}  & \cdots & \sqrt[A_1]{s_{g-1}\overline{s}_{g-1}} \\ 
\vdots & & \vdots  \\ \
 \sqrt[A_{g-1}]{ s_1\overline{s}_1}  & \cdots & \sqrt[A_{g-1}]{s_{g-1}\overline{s}_{g-1}} \\ 
\end{vmatrix}}.
\end{align*} 

In the following part, we reorganize the quotient in order to express it with some elementary functions. For that reason, we complete all the pairs of quadratic forms appearing in the matrices above to syzygetic tetrads as follows.

Let  $\{r_g,\overline{r}_g\}$ and $\{s_g,\overline{s}_g\}$ be any other two pairs of quadratic forms from ${\bf S}_{p_1+q_1}$ and ${\bf S}_{p_1+\overline{q}_1}$ such that $r_g+\overline{r}_g=p_1+q_1$ and $s_g+\overline{s}_g=p_1+\overline{q}_1$ different than any  $\{r_i,\overline{r}_i\}$ and $\{s_i,\overline{s}_i\}$ for $i=1,\dots , g-1$ respectively. 
Then we divide each row of the matrices by one of $\sqrt[T'_i]{r_g\overline{r}_g}$ and $\sqrt[T'_i]{s_g\overline{s}_g}$ with a suitable $T'_i$ among $A_1,\dots , A_{g-1}, B_1,\dots , B_{g-1}$ for each $i=2,\dots, 2g-3$. Hence we have

\begin{align}\label{eq:matkay1}
&\frac{\chi_{2,T'}(A_1)}{\chi_{1,T'}(A_1)} 
&=\frac
{d_{2}
\begin{vmatrix} 
\frac{\sqrt[B_1]{s_1\overline{s}_1}}{\sqrt[B_1]{s_g\overline{s}_g}}  & \cdots & \frac{\sqrt[B_1]{s_{g-1}\overline{s}_{g-1}}}{\sqrt[B_1]{s_g\overline{s}_g}} \\
\vdots & & \vdots \\
\frac{\sqrt[B_{g-1}]{s_1\overline{s}_1}}{\sqrt[B_{g-1}]{s_g\overline{s}_g}}   & \cdots & \frac{\sqrt[B_{g-1}]{s_{g-1}\overline{s}_{g-1}}}{\sqrt[B_{g-1}]{s_g\overline{s}_g}}
\end{vmatrix}
\begin{vmatrix} \frac{\sqrt[A_1]{ r_1\overline{r}_1}}{\sqrt[A_1]{ r_g\overline{r}_g}}  & \cdots & \frac{\sqrt[A_{1}]{ r_{g-1}\overline{r}_{g-1}}}{\sqrt[A_{1}]{ r_g\overline{r}_g}} \\ 
\vdots & & \vdots  \\ 
\frac{\sqrt[A_{g-1}]{ r_1\overline{r}_1}}{\sqrt[A_{g-1}]{ r_g\overline{r}_g}}  & \cdots & \frac{\sqrt[A_{g-1}]{ r_{g-1}\overline{r}_{g-1}}}{\sqrt[A_{g-1}]{ r_g\overline{r}_g}}
\end{vmatrix}}
{d_{1}\begin{vmatrix} \frac{\sqrt[B_1]{ r_1\overline{r}_1}}{\sqrt[B_1]{ r_g\overline{r}_g}}  & \cdots & \frac{\sqrt[B_{1}]{ r_{g-1}\overline{r}_{g-1}}}{\sqrt[B_{1}]{ r_g\overline{r}_g}} \\ 
\vdots & & \vdots  \\ 
\frac{\sqrt[B_{g-1}]{ r_1\overline{r}_1}}{\sqrt[B_{g-1}]{ r_g\overline{r}_g}}  & \cdots & \frac{\sqrt[B_{g-1}]{ r_{g-1}\overline{r}_{g-1}}}{\sqrt[B_{g-1}]{ r_g\overline{r}_g}}
\end{vmatrix}
 \begin{vmatrix} 
\frac{\sqrt[A_1]{s_1\overline{s}_1}}{\sqrt[A_1]{s_g\overline{s}_g}}  & \cdots & \frac{\sqrt[A_1]{s_{g-1}\overline{s}_{g-1}}}{\sqrt[A_1]{s_g\overline{s}_g}} \\
\vdots & & \vdots \\
\frac{\sqrt[A_{g-1}]{s_1\overline{s}_1}}{\sqrt[A_{g-1}]{s_g\overline{s}_g}}   & \cdots & \frac{\sqrt[A_{g-1}]{s_{g-1}\overline{s}_{g-1}}}{\sqrt[A_{g-1}]{s_g\overline{s}_g}}
\end{vmatrix}}, 
\end{align}

where  $d_{1}=\sqrt[B_1]{r_g\overline{r}_g}\cdots\sqrt[B_{g-1}]{r_g\overline{r}_g}\sqrt[A_1]{s_g\overline{s}_g}\cdots\sqrt[A_{g-1}]{s_g\overline{s}_g}$ and 
\\ \\$d_{2}=\sqrt[B_1]{s_g\overline{s}_g}\cdots \sqrt[B_{g-1}]{s_g\overline{s}_g}\sqrt[A_1]{r_g\overline{r}_g}\cdots\sqrt[A_{g-1}]{r_g\overline{r}_g}$.

\subsubsection{Computing the Quadrics}\label{compquad}

Now, all the entries of the four matrices in Equation~\eqref{eq:matkay1} are formed by a syzygetic tetrad of quadratic forms. Without loss of generality, we show how to obtain an elementary function by using such a tetrad only on the entries with $r_i, \overline{r}_i$'s.  
Among $2^{g-2}(2^{g-1}-1)-g$ pairs, we consider one more pair $\{r_{g+1},\overline{r}_{g+1}\}$ from ${\bf S}_{p_1+q_1}$ satisfying $r_{g+1}+\overline{r}_{g+1}=p_1+q_1$ 
different than any pairs $\{r_i,\overline{r}_i\}$ for $i=1,\dots , g$.
 Recall that the corresponding odd theta characteristic divisors to $r_i,\overline{r}_i$ are denoted by $D_{r_i}, D_{\overline{r}_i}$. 
It follows from \cite[Chapter 8]{Dol2012} that  
$D_{r_i}+D_{\overline{r}_i}+D_{r_{g+1}}+D_{\overline{r}_{g+1}}$ is cut out by a quadric $\calQ_{i}^r$ in $\PP^{g-1}.$ We can compute the quadric by computing the linear system of quadrics that pass through the points in the support of the aforementioned divisors. 
Denote the quadric $\calQ_{i}^s$ which exists for the tetrads among $s_i,\overline{s}_i$'s.

So, we have the following equalities between the following divisors on $\calC$
\begin{align*}
\Div\left(\frac{\sqrt{ r_i\overline{r}_i}}{\sqrt{ r_g\overline{r}_g}}\right)=\Div\left(\frac{\calQ_{i}^r}{\calQ_{g}^r}\right) \text{ \quad and \quad } \Div\left(\frac{\sqrt{ s_i\overline{s}_i}}{\sqrt{ s_g\overline{s}_g}}\right)=\Div\left(\frac{\calQ_{i}^s}{\calQ_{g}^s}\right) .
\end{align*}
It implies that there are constants $c_{r,ig},c_{s,ig}\in \CC$ such that 
\begin{align}\label{consquad}
\frac{\sqrt{ r_i\overline{r}_i}}{\sqrt{ r_g\overline{r}_g}}=c_{r,ig}\frac{\calQ_{i}^r}{\calQ_{g}^r} \text{ \quad and \quad } \frac{\sqrt{ r_i\overline{r}_i}}{\sqrt{ s_g\overline{s}_g}}=c_{s,ig}\frac{\calQ_{i}^s}{\calQ_{g}^s}. 
\end{align}

In the light of the computations above, we rewrite ${\chi_{2,T'}(A_1)}/{\chi_{1,T'}(A_1)}$ in terms of the corresponding quadrics and take the square of the quotient. Note that the constants in Equation~\eqref{consquad} appear in the numerator and denominator of the quotient ${\chi_{2,T'}(A_1)}/{\chi_{1,T'}(A_1)}$ in the same way, so they are cancelled out. We do not include them in the quotient $\left(\frac{\chi_{2,T'}(A_1)}{\chi_{1,T'}(A_1)}\right)^2$.  Therefore, Equation~\eqref{sadelesme} implies Theorem~\ref{mainTheorem}.

\begin{remark}\label{QuaSte} With respect to the proof, we need not only to compute the Steiner sets ${\bf S}_{p_1+q_1}$ and ${\bf S}_{p_1+\overline{q}_1}$ in terms of the multitangents but also label them with respect to a symplectic basis. Namely, we need to compute a complete 2-level structure of the curve. We refer to \cite[Algorithm 3.2]{CelKulRenSay2018} to obtain all the Steiner sets from the set of multitangents, which uses the geometric characterization of being syzygetic as we use in Section~\ref{compquad}. One can compute the multitangents with some algebraic geometric methods involving Gröbner bases, resultants etc. For the latter, we remind Corollary~\ref{SymSte}. Nevertheless, we need to point out that neither computing the Steiner sets which is equivalent to computing of the quadrics described in Section~\ref{compquad} nor computing the multitangents is generally a low-cost task from the computational perspective.   
\end{remark}  

Thanks to Theorem~\ref{mainTheorem}, we establish Algorithm~\ref{AlgTheCon}. We firstly give a preliminary algorithm in order to prepare the contents of Algorithm~\ref{AlgTheCon}. \\

   \begin{algorithm}[H]
\caption{Preliminary Computation}\label{PreAlgo}
\begin{algorithmic}[1]
  \REQUIRE 
 \text{ } \\
$\calC$, canonical model of the curve in $\PP^{g-1}$. \\
 \ENSURE $\mathfrak{S}$, set of all Steiner sets in which all the multitangents labeled with appropriate characteristics with respect to a symplectic basis.
\STATE Compute the multitangents. 
\STATE Compute the Steiner Sets. 
\STATE Label all the multitangents as described in Corollary~\ref{SymSte}.  
\RETURN $\mathfrak{S}$.
    \end{algorithmic}
\end{algorithm}

   \begin{algorithm}[H]
\caption{Theta Constants}\label{AlgTheCon}
\begin{algorithmic}[1]
  \REQUIRE 
 \text{ } \\
\begin{itemize}
\item $\calC$, canonical model of the curve in $\PP^{g-1}$. \\
\item Even characteristics $p_1,p_2\in \ZZ_2^g\oplus \ZZ_2^g$ with respect to the labeling obtained by Algorithm~\ref{PreAlgo}.
\end{itemize}
  \ENSURE $\left(\frac{\vartheta[p_1](\tau)}{\vartheta[p_2](\tau)}\right)^4$.
\STATE Set the Steiner sets ${\bf S}_{p_1+q_1}$ and ${\bf S}_{p_1+\overline{q}_1}$ with $q_1+\overline{q}_1=p_1+p_2$. 
\STATE \label{stepp} Set randomly $g+1$ many $r_i,\overline{r}_i$ and $s_i,\overline{s}_i$ in ${\bf S}_{p_1+q_1}$ and ${\bf S}_{p_1+\overline{q}_1}$ respectively for $i=1,\dots , g+1$. 
\STATE Call the multitangents $\beta_{q_1},\beta_{\overline{q}_1}$ and compute the contact points $A_1,\dots , A_{g-1}$ and $B_1,\dots , B_{g-1}$ of $\beta_{q_1}$ and $\beta_{\overline{q}_1}$ with $\calC$ respectively.
\STATE \label{steppp} Call the multitangents $\beta_{r_i}, \beta_{\overline{r}_i}$ and  $\beta_{s_i}, \beta_{\overline{s}_i}$ for $i=1,\dots , g+1$ and compute their contact points with $\calC$.
\STATE Now compute the quadrics $\calQ^r_{i}, \calQ^r_{g}$ and $\calQ^s_{i}, \calQ^s_{g}$ for $i=1,\dots ,g-1$ via the tangency points computed in Step~\ref{steppp}. 
\STATE Check whether $\calQ^r_{i}$'s and $\calQ^s_{i}$'s are linearly independent for $i=1,\dots , g-1$ separately. If one of them fails to be linearly independent then restart the procedure from Step~\ref{stepp}.
\STATE\label{Step} Compute $d_1$ and $d_2$.
\RETURN $\left(\frac{\chi_{2,T'}(A_1)}{\chi_{1,T'}(A_1)}\right)^2$.
    \end{algorithmic}
\end{algorithm} 
\text{ } \\

The algorithm has been implemented in MAGMA \cite{magma} and available on \url{https://turkuozlum.wixsite.com/tocj}. In the code file, there are supplemental codes computing a complete 2-level structure of $\calC$ when $\calC$ is a non-hyperelliptic curve of genus 4 lying on a quadric cone. 

\begin{remark} Step~\ref{Step} in Algorithm~\ref{AlgTheCon} has a conditional statement. This situation arises from the choices in Equation~\eqref{choiceBasis}. Indeed, these choices can rarely fail to be the bases of $\mathcal{L}(D_{p_1}+\kappa)$ or $\mathcal{L}(D_{p_2}+\kappa)$. However we can control it computationally by checking the linear dependence of the quadrics forming the quotient  $\left(\frac{\chi_{2,S'}(A_1)}{\chi_{1,S'}(A_1)}\right)^2$. We actually anticipate the following conjecture holds. 
\end{remark}

\begin{conjecture}\label{Que}
For a fixed $v\in \Pic(\calC)[2]$, the following map
\begin{align*}
\bigoplus_{\substack{D, D+v }}\calL(D)\otimes \calL(D+v) \longrightarrow \calL(\kappa+v)
\end{align*}
surjective when $D, D+v$ run through all the effective and odd theta characteristic divisors.
\end{conjecture}

\section{Applying the Algorithm in genus 4}\label{DelPezzoSubsections}

In this section, we apply the algorithm on an example of a curve of genus 4. We avoid the case of genus 3 to use for the application since Weber's formula given in Equation~\eqref{WeberEng} does the job. However, %as a first step to apply our method, we have used Weber's formula.
we shall note that 
our formula and Weber's formula coincide on all the examples we have tried as we have expected, because we obtain our formula from Weber's formula immediately after several algebraic computations applied to it.

Now assume that the genus of $\calC$ is 4. The canonical model is given by the complete intersection of a quadric $\calQ$ and a cubic surface $\mathcal{R}$ in $\PP^3$. We call such a curve a \emph{space sextic}. A space sextic lies on either a smooth quadric or a quadric cone.
 For the application, we focus on the ones lying on a quadric cone since they provide a relatively efficient way to obtain a complete 2-level structure.

\subsection{Space Sextics Lying on a Singular Quadric}

Assume that $\calQ$ is a quadric cone. In this case, $\calC$ has a vanishing even theta constant. The characteristic of this constant is corresponding to the unique effective even theta characteristic divisor for which the dimension of the Riemann-Roch space is 2. The corresponding extrinsic geometric object is a one dimensional family of tritangents. Each plane in this family passes through the node of $\calQ$. In the meantime, for each effective odd theta characteristic divisor, there is a unique tritangent. There are 120 of them.  

\subsubsection{The Canonical Model and Tritangents} Such a space sextic is constructed from a del Pezzo surface of degree 1 \cite[Proposition 5.1]{CelKulRenSay2018}. Its geometry is more transparent compared to the other case thanks to the geometric structure of the del Pezzo surfaces. 

Let $\calS$ be a del Pezzo surface of degree 1. The surface $\calS$ is isomorphic to the blow up of $\PP^2$ at 8 points in general position, say $P_1,\dots, P_8$. We denote $\mathcal{P}=\{P_1,\dots, P_8\}$. 
We treat $\calS$ as the blow up. The anticanonical model for $\calS$ is a sextic hypersurface in $\PP(1:1:2:3)$. We can compute this model by starting with $P_1,\dots , P_8$.  If we consider the projection $\pi : \PP(1:1:2:3) \rightarrow \PP(1:1:2)$ then it is generically 2-1 branched along a curve $\calC'$. If we embed $\PP(1:1:2)$ as a singular quadric surface in $\PP^3$ by $\phi$ then the image $\calC:=\phi(\calC)$ under $\phi$ is a curve of genus 4. Let $\psi$ be the blow up map then we have the following diagram. 

\begin{figure}[h]
  \centering
  \begin{tikzpicture}
    \node (P2) {$\PP^2$};
    \node[anchor=base west,xshift=1cm] (P1123) at (P2.base east) {$\PP(1\!:\!1\!:\!2\!:\!3)$};
    \node[anchor=base west,yshift=-25mm] (P112) at (P1123.base west) {$\PP(1\!:\!1\!:\!2)$};
    \node[anchor=base west,yshift=-25mm] (P3) at (P112.base west) {$\PP^3$};
    \draw[->] (P2) -- node[above] {$\psi$} (P1123);
    \draw[->] ($(P1123.south west)+(0.25,0)$) -- node[left] {$\pi$} node[right]
    {
      \begin{tikzpicture}[every node/.style={font=\scriptsize}]
        \node (P1123elem) {$(x\!:\!y\!:\!z\!:\!w)$};
        \draw[|->] ($(P1123elem.south west)+(0.75,0)$) -- ++(0,-0.5);
        \node[anchor=base west,yshift=-10mm] at (P1123elem.base west) {$(s\!:\!t\!:\!w)$};
      \end{tikzpicture}
    }
    ++(0,-1.7);
    \draw[->] ($(P112.south west)+(0.25,0)$) -- node[left] {$\phi$} node[right]
    {
      \begin{tikzpicture}[every node/.style={font=\scriptsize}]
        \node (P112elem) {$(x\!:\!y\!:\!z)$};
        \draw[|->] ($(P112elem.south west)+(0.75,0)$) -- ++(0,-0.5);
        \node[anchor=base west,yshift=-10mm] at (P112elem.base west) {$(x^2\!:\!xy\!:\!y^2\!:\!z)$};
      \end{tikzpicture}
    }
    ++(0,-1.7);
    \node[anchor=base west,xshift=10mm] (X) at (P1123.base east) {$\calS:=\overline{\psi(\mathbb P^2)}$};
    \draw[draw opacity=0] (X) -- node[sloped] (supsetTop) {$\supset$} (P1123);
    \node[anchor=base west,yshift=-25mm] (Cprime) at (X.base west) {$\calC':=\text{BranchCurve}(\pi_\calS)$};
    \node[anchor=base,yshift=-25mm] (supsetMid) at (supsetTop.base) {$\supset$};
    \node[anchor=base west,yshift=-25mm] (C) at (Cprime.base west) {$\calC$};
    \node[anchor=base,yshift=-25mm] (supsetLow) at (supsetMid.base) {$\supset$};
    \draw[->] ($(Cprime.south west)+(0.25,-0.1)$) -- ++(0,-1.6);
  \end{tikzpicture}
  \caption{The del Pezzo surface $\calS$ of degree $1$ and the branch curve $\mathcal{C}'$.}
  \label{fig:diagram1}
\end{figure}

This disposition enables us to compute the defining equations of $\calC$ and the equations of the tritangents starting with $\mathcal{P}$. 
Indeed, the map between $\calS$ and $\calC$ gives a 2-1 correspondence between the exceptional curves on $\calS$ and tritangents of $\calC$. The exceptional curves are the images of the following pairs in $\PP^2$ denoted by (0,6), (1,5), (2,4), (3,3) respectively. 

\begin{enumerate}
	\item[(0,6)]
	The point $P_i$ and the sextic vanishing triply at $P_i$ and doubly at the other seven points.
	\item[(1,5)]
	The line through $\{P_i, P_j\}$ and the quintic vanishing at all eight points and doubly at the six points in $\mathcal{P} \backslash \{P_i, P_j\}$.
	\item[(2,4)]
	The conic through $\mathcal P\setminus \{P_i,P_j,P_k\}$ and the quartic vanishing at $\mathcal{P}$ and doubly at $P_i,P_j,P_k$.
	\item[(3,3)]
	The cubic vanishing doubly at $P_i$, non-vanishing at  $P_j$, and vanishing singly at $\mathcal{P} \backslash \{P_i,P_j\}$ and the cubic vanishing doubly at $P_j$, non-vanishing at $P_i$, and vanishing singly at $\mathcal{P} \backslash \{P_i,P_j\}$.
\end{enumerate}

For a detailed explanation, we refer \cite{CelKulRenSay2018}. 

Moreover, we do not have only the equations but also a complete 2-level structure of $\calC$ as follows. 

\subsubsection{Labeling} The configuration of exceptional curves on $\calS$ enables us to label the tritangents with appropriate characteristics in a coordinate-free way. In the following part, we show how to obtain a labeling which is explained in Section~\ref{labelling}

Let $E_1, \dots , E_8$ be the exceptional divisors lying above $P_1,\dots , P_8$ under the blow-up map and $\kappa_\calS$ be the canonical divisor of $\calS$. Suppose that $\rho : \Pic \calS \rightarrow \Pic \calC$ be the natural restriction homomorphism. Let $\langle \cdot , \cdot \rangle $ be the Weil pairing on $\Pic(\calC)[2]$. Set $v_i:= \rho(E_i+\kappa_\calS) $ and $v_9=\sum_{i=1}^8$ then $\{v_1,\dots ,v_9\}$ is a fundamental set of $\Pic(\calC)[2]$ i.e. $\langle \rho(E_i+\kappa_\calS) , \rho (E_j , \kappa_\calS) \rangle = 1$ if $i\neq j$ which follows from \cite[Theorem 2.1]{Zarhin2008}. Now, we consider $\rho(-\kappa_S)$. It is an even theta characteristic divisor by \cite[Lemma 2.4(ii)]{Zarhin2008}. We take the quadratic form $q:=q_{\rho(-\kappa_\calS)}$ corresponding to the theta characteristic divisor $\rho(-\kappa_\calS)$. We define $q_i:=q+v_i+v_9$ for $i=1,\dots , 9$. It follows from Proposition~\ref{prop:Aron} that the set $\{q_1,\dots , q_9\}$ forms an Aronhold basis of $\Pic(\calC)[2]\bigsqcup Q\Pic(\calC)[2]$. Hence, we can express all the quadratic forms in terms of $q_1,\dots , q_9$ and label them via subsets of \{1,\dots ,9\} which is formed by the indices of the points $P_1,\dots , P_8$ and an extra index 9 as follows.

\begin{center}
\begin{tabular}{c l c c r }
 
even &  & $q_1,\dots , q_9 $&\qquad$\longleftrightarrow$ & \qquad $\{i\}$ \\
odd & & $q_i+q_j+q_k$&\qquad$\longleftrightarrow$ & \qquad$\{i,j,k\}$  \\ 
even &  & $q_{i_1}+\dots+q_{i_5}$&\qquad$\longleftrightarrow$  &\qquad$\{i_1,\dots , i_5\}$ \\ 
odd & & $q_{i_1}+\dots+q_{i_7} $&\qquad$\longleftrightarrow$  &\qquad$\{i_1,\dots , i_7\}$ \\
even &  & $q_{1}+\dots+q_{9}$&\qquad$\longleftrightarrow$  &\qquad$\{1,\dots , 9\}$
\end{tabular}
\end{center} 

This labeling determines the parities of quadratic forms depending on only the cardinality of the label by definition of Aronhold basis and distinguishes them from each other, which follows from Remark~\ref{determination}. 

Furthermore, we can specify which quadratic form corresponds to which exceptional divisor on $\calS$ by finding the corresponding theta characteristic divisor. We give the following table for the correspondence and refer \cite[Section 1.2.3]{Tur2018} for the computations.

\begin{table}[H]
 \thisfloatpagestyle{empty}
  \begin{tabular}{l c  r }
   $D_{q_{ijk}}$ \qquad& \qquad$\longleftrightarrow$ \qquad & \qquad (2,4) \\ 
      $D_{q_{ij9}}$ \qquad&  \qquad$\longleftrightarrow$ \qquad & \qquad (3,3) \\ 
         $D_{q_{i_1\dots i_7}}$ \qquad& \qquad$\longleftrightarrow$ \qquad& \qquad (0,6) \\ 
          $D_{q_{i_1\dots i_6 9}}$ \qquad & \qquad $\longleftrightarrow$  \qquad& \qquad (1,5) \\
  \end{tabular}
\end{table}

\begin{example}\label{Ex} Let $k=\CC$. We consider the following 8 points in $\PP^2(k)$. 
 \begin{align*}
    P_1& = (2\!:\!-3\!:\!1), &P_5&=(1/2\!:\!0\!:\!1),&
    P_2& = (3/2\!:\!1\!:\!1), &P_6&=(2/3\!:\!2/3\!:\!1),\\
    P_3& = (0\!:\!-3/2\!:\!1), &P_7&=(-3\!:\!2\!:\!1),&
    P_4& = (3/2\!:\!3/2\!:\!1), &P_8&=(1/3\!:\!1\!:\!1).
  \end{align*}

The defining equations of the curve $\calC$ of genus 4 are the following equations.\\

\begin{tikzpicture}
\node[anchor=west,  scale=0.020cm] (eqn) {\begin{tabular}{l} $x_0^3 + 2425564030663/162140107530x_0^2x_1$ \\
$+ 15669691012720998280286400529/149429031846570347991915600x_0^2x_2$ \\
$+ 211709448479418431107937289647/448287095539711043975746800x_0x_1x_2$\\ 
$+   9910047994802558384716635818134607/7644191553143152721874434433600x_0x_2^2$\\ 
$+ 905783995186184025726770668993123/395389218266025140786608677600x_1x_2^2$\\ 
$+ 92368472409963092742435769596441128153/55863751870370160091458366840748800x_2^3$ \\
    $- 29830449072973532706572819/236073482149263407485060x_0^2x_3$\\ 
    $- 1011526537873326300399754441310551/637015962761929393489536202800x_0x_1x_3$\\
    $ - 36690171700015844095035636292739030279/5431198098508210008891785665072800x_0x_2x_3$\\ 
    $-  331206537690486038353064612628955433/17557752473625678908055341589675x_1x_2x_3 $\\
    $- 1062182808186693286084865785168875289650763/79382391407795997489962339280704044800x_2^2x_3 $\\
    $+ 
    1162135044692003397844695454373200491311/1715051666217814760585606095570766400x_0x_3^2 $\\
    $+ 24403529360323237608442590572012473543277/798386120480706871307092492765701600x_1x_3^2 $\\
    $- 
    2260119208090704069704104349128528879254888149/112802378190478112433236484117880447660800x_2x_3^2$\\ 
    $+ 
    13204897029090296036812705006873263119397966616189/53430726469556465922543014643836038708665600x_13^3,$ \\
$x_1^2 - x_0x_2$ \end{tabular}} ; 
\end{tikzpicture}

The list of the equations of the tritangents and their corresponding labels are computed as follows. \\

\begin{tikzpicture}
\node[xshift= 10cm, scale=0.016cm] (n1) {\begin{tabular}{l}
    $<x_0 + 4073720176917559726/1133376368146185855x_1 + 327453011960962204578454/24157917287035951499325x_2 - 112735162688749958708129581/1373136018595123483221633x_3, \{ 1, 2, 4, 5, 6, 7, 8 \}>,$\\
    $<x_0 + 350578164144253479/46203319572275680x_1 + 18637999053114537307081/2363577016039334686080x_2 + 131316444954786724427554199/3358642939791894588919680x_3, \{ 1, 2, 3, 4, 5, 6, 8 \}>,$\\
    $<x_0 + 211479792266371/218188348090920x_1 + 1316618230054501981871/70690406521280988960x_2 - 5757723440908096949483149/33483689222246761770720x_3, \{ 1, 2, 3, 4, 5, 7, 8 \}>,$\\
    $<x_0 + 17373590988118142/4158176266996215x_1 + 214542061297728791377/17726305426204864545x_2 - 1510026721794434231897344/25189080010637112518445x_3, \{ 1, 2, 3, 5, 6, 7, 8 \}>,$\\
    $<x_0 + 4864890585953831/378366090117870x_1 - 512579666106105329/131671399361018760x_2 + 18527122874371793828509/80187882210860424840x_3, \{ 1, 2, 3, 4, 5, 6, 7 \}>,$\\
    $<x_0 + 149286535172373/30708550083235x_1 + 964354660076388445/89767233603312552x_2 - 7710930391228068766751/212598731583845227320x_3, \{ 1, 3, 4, 5, 6, 7, 8 \}>,$\\
    $<x_0 + 600953315525715/190786792110254x_1 + 279475735717785218413/19519778274384307248x_2 - 4469565218503206087488519/46229341546500167665680x_3, \{ 1, 2, 3, 4, 6, 7, 8 \}>,$\\
    $<x_0 + 24850715371714303/5696394486025210x_1 + 3820203557552276572807/291404756327105642760x_2 - 786876715268942791133399/12548065416388397526120x_3, \{ 2, 3, 4, 5, 6, 7, 8 \}>,$\\
    $<x_0 + 903225497/369556365x_1 + 199224827801747/12603350271960x_2 - 6542142470614222261/53728082209365480x_3, \{ 1, 2, 4, 6, 7, 8, 9 \}>,$\\
    $<x_0 + 1289061985/307910162x_1 + 209545404053987/15751452247272x_2 - 2527971780098556301/37304689405622520x_3, \{ 1, 2, 4, 5, 6, 8, 9 \}>,$\\
    $<x_0 + 3227297803/826483812x_1 + 90490251893963/7046600981112x_2 - 14336323894698807/202287272609296x_3, \{ 1, 2, 5, 6, 7, 8, 9 \}>,$\\
    $<x_0 + 57335427860129/11475325450080x_1 + 679202562690930247/65225749858254720x_2 - 25978847750809076598431/834172114937219614080x_3, \{ 2, 3, 5, 6, 7, 8, 9 \}>,$\\
    $<x_0 + 6322191013177/3814460277285x_1 + 935124327549320243/65044176648263820x_2 - 3983438390860063284061/30809258339060962740x_3, \{ 1, 3, 4, 5, 7, 8, 9 \}>,$\\
    $<x_1 - 12646609161011223/21350501266145380x_2 + 3321248368938556955/130959981147593892x_3, \{ 1, 3, 4, 5, 6, 7, 9 \}>,$\\
    $<x_0 + 5122998713/662005752x_1 + 96913679546351/11288522083104x_2 + 8883742040081694071/240614848201361760x_3, \{ 2, 4, 5, 6, 7, 8, 9 \}>,$\\
    $<x_0 + 1935440839217/237125656815x_1 + 36305614923105853/4043466700009380x_2 + 264987832094981182319/5745766180713328980x_3, \{ 1, 3, 4, 5, 6, 8, 9 \}>,$\\
    $<x_0 + 3663575555203/986415103560x_1 + 226684261333339967/16820350345905120x_2 - 209898118142285233287/2655746426836797280x_3, \{ 1, 2, 3, 4, 7, 8, 9 \}>,$\\
    $<x_0 - 1152767551702/459889427085x_1 + 7025018642918134/280072661094765x_2 - 1161242235470603367/4020032842582435x_3, \{ 2, 3, 4, 5, 7, 8, 9 \}>,$\\
    $<x_0 + 19960603357/5537568230x_1 + 20324578773181561/1537804847743920x_2 - 174334477334169059389/2185220688644110320x_3, \{ 1, 4, 5, 6, 7, 8, 9 \}>,$\\
    $<x_0 + 12997908343344/3072843440995x_1 + 1182489526791012001/89825359467165840x_2 - 8415082833189358503901/127641835802842658640x_3, \{ 1, 3, 5, 6, 7, 8, 9 \}>,$\\
    $<x_0 + 27526310486771/6714783636045x_1 + 31005944979969511/2336744705343660x_2 - 32920978436860027141/474359175184762980x_3, \{ 1, 2, 3, 4, 6, 7, 9 \}>,$\\
    $<x_0 + 601775140562/121781832033x_1 + 76094874659365/5967309769617x_2 - 210266083788915836/4385972680668495x_3, \{ 1, 2, 3, 4, 5, 7, 9 \}>,$\\
    $<x_0 + 12989495751/3553066880x_1 + 1343532205225249/98670088484352x_2 - 4778774631566269/58847559697920x_3, \{ 1, 2, 4, 5, 6, 7, 9 \}>,$\\
    $<x_0 + 236687709629/45883182720x_1 + 5324639502755791/434666684300800x_2 - 8883742040081694071/222358089020917248x_3, \{ 1, 2, 3, 4, 6, 8, 9 \}>,$\\
    $<x_0 + 240885426739/27411960735x_1 + 332713529569297/80591164560900x_2 + 4778774631566269/52652894179788x_3, \{ 2, 3, 4, 6, 7, 8, 9 \}>,$\\
    $<x_0 + 1496363111081/538417805370x_1 + 276048316628067587/18362200834338480x_2 - 2853783855704112955151/26092687385594980080x_3, \{ 2, 3, 4, 5, 6, 7, 9 \}>,$\\
    $<x_0 + 2108609486/598552095x_1 + 11620934086627/850542526995x_2 - 866406732068672/10271565984645x_3, \{ 2, 3, 4, 5, 6, 8, 9 \}>,$\\
    $<x_0 + 520021597/85223190x_1 + 365561073070657/32801383152720x_2 - 578231730419518549/46610765460015120x_3, \{ 1, 2, 3, 4, 5, 8, 9 \}>,$\\
    $<x_0 + 268252789634/41705786505x_1 + 1852484436921866/177791767870815x_2 - 14336323894698807/28071344682714235x_3, \{ 1, 2, 3, 4, 5, 6, 9 \}>,$\\
    $<x_0 + 50327099887/14925294618x_1 + 3553120080354109/254506123826136x_2 - 97195497231426345191/1084959605870817768x_3, \{ 1, 2, 4, 5, 7, 8, 9 \}>,$\\
    $<x_0 + 1352622552633/107696545910x_1 + 22683321831250721/5509324502571960x_2 + 450394730250489286981/2609583372718251720x_3, \{ 1, 2, 3, 5, 6, 8, 9 \}>,$\\
    $<x_0 + 650910009649/206892503130x_1 + 5375731423734173/391992329263640x_2 - 1396639895046926211671/15039569696858075880x_3, \{ 1, 2, 3, 5, 7, 8, 9 \}>,$\\
    $<x_0 + 89148743916331/20793229551720x_1 + 4207062252255625279/354566150315929440x_2 - 9406582993746727996021/167946166532978578080x_3, \{ 3, 4, 5, 6, 7, 8, 9 \}>,$\\
    $<x_0 + 3714576415226/1043070075645x_1 + 692159506850267/51629697909720x_2 - 12433267694395540129861/151647110108314952040x_3, \{ 1, 2, 3, 6, 7, 8, 9 \}>,$\\
    $<x_0 + 10524782449/6140033768x_1 + 173639060741794993/11801169462231072x_2 - 4778774631566269/36733616885280x_3, \{ 1, 2, 3, 5, 6, 7, 9 \}>,$\\
    $<x_0 + 76000541890202/23809171477365x_1 + 196599415703466254/14499785429715285x_2 - 624542835374027261879/6868065031875139995x_3, \{ 1, 3, 4, 6, 7, 8, 9 \}>,$\\
    $<x_0 + 7308689401/3659160240x_1 + 148488847558127/8913714344640x_2 - 578231730419518549/4222129361244480x_3, \{ 6, 7, 8 \}>,$\\
    $<x_0 + 7838339/851180x_1 + 3149047901/410782680x_2 + 4778774631566269/61874551957680x_3, \{ 1, 2, 8 \}>,$\\
    $<x_0 + 46606283/19508288x_1 + 14017818625949/997965980928x_2 - 52566520947228959/472703219632896x_3, \{ 2, 5, 6 \}>,$\\
    $<x_0 + 909136639/136867830x_1 + 57232887380099/4667740474320x_2 - 4778774631566269/602987201273520x_3, \{ 1, 5, 7 \}>,$\\
    $<x_0 + 8676718411/2339134650x_1 + 5212099125213067/398869240518000x_2 - 578231730419518549/7557242543681040x_3, \{ 4, 6, 7 \}>,$ \\
     $<x_0 + 4031771577/1564767490x_1 + 158517081960157/11435320816920x_2 - 1725137641995423109/16249590880843320x_3, \{ 3, 4, 7 \}>,$\\
    $<x_0 + 389217154/47969625x_1 + 167025712745806/19426978580625x_2 + 4778774631566269/100384496592975x_3, \{ 1, 4, 6 \}>,$\\
    $<x_0 + 26010483/4488410x_1 + 5324595000427/459218203920x_2 - 4778774631566269/217516355923440x_3, \{ 2, 3, 4 \}>,$\\
    $<x_0 + 460241374/92711655x_1 + 1765611173520406/142677952480665x_2 - 15191724553749169151/337908950791708275x_3, \{ 3, 6, 8 \}>,$\\
    $<x_0 + 20785293/2673025x_1 + 2494820286557/273482533800x_2 + 4778774631566269/129539560176600x_3, \{ 3, 4, 5 \}>,$\\
    $<x_0 + 1294221563/326271135x_1 + 24858158126371/1854525131340x_2 - 578231730419518549/7905840634902420x_3, \{ 2, 4, 6 \}>,$\\
    $<x_0 + 118176714/25240405x_1 + 28836334859897/2582396316360x_2 - 52566520947228959/1223195055182520x_3, \{ 2, 7, 8 \}>,$\\
    $<x_0 - 6045004297/111280560x_1 + 2549009786513/65432969280x_2 - 4778774631566269/3206215494720x_3, \{ 1, 2, 3 \}>,$\\
    $<x_0 + 80500991/19668981x_1 + 5337932826827/399280314300x_2 - 4778774631566269/68085279194436x_3, \{ 1, 3, 5 \}>,$\\
    $<x_0 + 14315647526/3533406315x_1 + 186465983111326/15062911120845x_2 - 1381065868522651741/21404396702720745x_3, \{ 4, 7, 8 \}>,$ \\
           $<x_0 + 29240926201/8544126720x_1 + 507812636527649/38038452157440x_2 - 52566520947228959/621294718571520x_3, \{ 4, 5, 7 \}>,$\\
    $<x_0 + 515020940407/108383951640x_1 + 252792229397791/19703231805600x_2 - 1314163023680723975/25011807873543456x_3, \{ 3, 5, 7 \}>,$ \\
    $<x_0 + 2801718062194/700321840155x_1 + 635693192925161299/47767552073292240x_2 - 1624252930748427604141/22625897165382757680x_3, \{ 3, 6, 7 \}>,$\\

    $<x_0 + 890089557298/179791690485x_1 + 9971686829623163/766451976537555x_2 - 4893465222723859456/99011659878169605x_3, \{ 1, 3, 6 \}>,$\\
    $<x_0 + 6987394176/2124051725x_1 + 6122731721078153/434631960176400x_2 - 18976514061949654199/205870671803554800x_3, \{ 5, 7, 8 \}>,$ \\
        $<x_0 + 6044782829/771271170x_1 + 240792101418047/26303431981680x_2 + 157699562841686877/4153019649551920x_3, \{ 1, 4, 8 \}>,$\\

    $<x_0 + 290811137/159178200x_1 + 29208827521823/1938790476000x_2 - 4778774631566269/36733616885280x_3, \{ 2, 3, 8 \}>,$\\
    $<x_0 + 3794287364/624085305x_1 + 543742094766601/42567610483440x_2 - 157699562841686877/6720952721885360x_3, \{ 1, 3, 8 \}>,$\\
    $<x_0 + 252777853/70107910x_1 + 97913108091947/7172880487920x_2 - 52566520947228959/637041448333395x_3, \{ 5, 6, 7 \}>,$\\
    $<x_0 + 1044954439/123675750x_1 + 3763015334623/421783777800x_2 + 4778774631566269/90811325493000x_3, \{ 1, 2, 5 \}>,$\\
    $<x_0 + 22136253442/3020106585x_1 + 174544232379842/12874714371855x_2 - 4778774631566269/6098323040801985x_3, \{ 1, 3, 7 \}>,$\\
    $<x_0 + 37586733/11415598x_1 + 1130973706147/83425190184x_2 - 52566520947228959/592735976257320x_3, \{ 2, 5, 8 \}>,$\\
    $<x_0 + 146925169/22158645x_1 + 8478597050327/755698429080x_2 - 4778774631566269/3221542403168040x_3, \{ 1, 5, 8 \}>,$\\
    $<x_1 - 73383553009/92215766580x_2 + 52566520947228959/1965579064652700x_3, \{ 3, 4, 8 \}>,$\\
    $<x_0 + 11923160198/886123995x_1 + 27276319349773/3777546590685x_2 + 305841576420241216/1789297901787795x_3, \{ 1, 4, 7 \}>,$\\
    $<x_0 + 314590005619/970111005x_1 - 1663113769574003/5514110952420x_2 + 208741654681446196189/23506654990166460x_3, \{ 2, 3, 6 \}>,$\\
    $<x_0 + 31952261/7825302x_1 + 158700294619/12708290448x_2 - 52566520947228959/812631632697360x_3, \{ 2, 6, 7 \}>,$\\
    $<x_0 + 219292391/121034120x_1 + 12682688588981/884517348960x_2 - 52566520947228959/418966384290720x_3, \{ 1, 3, 4 \}>,$\\
    $<x_0 + 8751112709/1344617670x_1 + 36568751907253/3275488644120x_2 - 4778774631566269/1269400735443960x_3, \{ 1, 6, 8 \}>,$\\
    $<x_0 + 3966334334/1192648455x_1 + 68360391333989/5084260363665x_2 - 210266083788915836/2408244658922655x_3, \{ 4, 6, 8 \}>,$\\
    $<x_0 + 117621798/19513285x_1 + 869684231805451/75864842166960x_2 - 52566520947228959/3266786082401520x_3, \{ 1, 5, 6 \}>,$\\
    $<x_0 + 860258317/134539608x_1 + 23549376325625/2294169395616x_2 - 4778774631566269/16300073555851680x_3, \{ 1, 2, 4 \}>,$\\
    $<x_0 + 96334320299/15112190400x_1 + 2678681168397901/257693070700800x_2 - 578231730419518549/366181853465836800x_3, \{ 3, 5, 6 \}>,$\\
      $<x_0 - 1752631/1813911x_1 + 340452553667/20620540248x_2 - 4778774631566269/23974189930152x_3, \{ 1, 2, 7 \}>,$\\
    $<x_0 + 5014787462/590751915x_1 + 1574752666741/193721185665x_2 + 210266083788915836/3578611462789545x_3, \{ 2, 3, 5 \}>,$\\
      
    $<x_0 + 185687023/51000960x_1 + 1685604014111/124238338560x_2 - 4778774631566269/58847559697920x_3, \{ 2, 4, 8 \}>,$\\
    $<x_0 + 182494387/57564237x_1 + 5688366094843/396811106748x_2 - 1419296065575181893/14723235214654820x_3, \{ 2, 4, 7 \}>,$\\
    $<x_0 + 3860992291/221122095x_1 - 25760867/161321780x_2 + 4778774631566269/15620949279180x_3, \{ 1, 4, 5 \}>,$\\
    $<x_0 + 737566882/156964377x_1 + 8749351307498/669139139151x_2 - 52566520947228959/950846716733571x_3, \{ 2, 4, 5 \}>,$\\
    $<x_0 + 97152207638/26538712935x_1 + 1528183807691977/113134533241905x_2 - 12921806603755191376/160764171736747005x_3, \{ 3, 7, 8 \}>,$\\
        $<x_0 + 869455681/96464070x_1 + 13998956072243/1644905321640x_2 + 52566520947228959/779136820683480x_3, \{ 3, 4, 6 \}>,$\\
    $<x_0 + 3439954626/440890307x_1 + 4489039676461/495696363624x_2 + 52566520947228959/1387424079579720x_3, \{ 1, 2, 6 \}>,$\\
    
       $<x_0 + 19942511/5396760x_1 + 4607902701299/349697095776x_2 - 578231730419518549/7453793596465440x_3, \{ 2, 5, 7 \}>,$\\
    $<x_0 - 459505759/75620760x_1 + 22567415192149/1289485199520x_2 - 578231730419518549/1832358468517920x_3, \{ 4, 5, 6 \}>,$\\
    $<x_0 + 1333682313/239656120x_1 + 153968572493101/12259848474720x_2 - 52566520947228959/1583749516597920x_3, \{ 1, 6, 7 \}>,$\\
    $<x_0 + 542447699/149430855x_1 + 67811165515417/5096189878920x_2 - 578231730419518549/7241685817945320x_3, \{ 2, 6, 8 \}>,$\\

    $<x_0 + 5759307737/1413139038x_1 + 642621576031363/48193693751952x_2 - 14455793260487963725/205449716464571376x_3, \{ 5, 6, 8 \}>,$\\
    $<x_0 + 17374741849/2923824405x_1 + 559830400133891/49857053754060x_2 - 129026915052289263/7871874820502140x_3, \{ 3, 5, 8 \}>,$ \\

    $<x_0 + 15558397/3409786x_1 + 2272081472671/174431012616x_2 - 4778774631566269/82622156309112x_3, \{ 1, 7, 8 \}>,$\\
    $<x_0 + 649183624679/193587931455x_1 + 44772153535010761/3301061407170660x_2 - 45575173661247507453/521200917732167540x_3, \{ 2, 3, 7 \}>,$\\
    $<x_0 + 2825601271/1201758249x_1 + 283409432125067/20492381661948x_2 - 358408097367470175/3235519371292012x_3, \{ 4, 5, 8 \}>,$\\

    $<x_0 + 452954149/55364130x_1 + 2453316756031/269734041360x_2 + 52566520947228959/1149876218317680x_3, \{ 4, 5, 9 \}>,$\\
    $<x_0 + 4153381142/1156780515x_1 + 22360431417194/1643785111815x_2 - 52566520947228959/637041448333395x_3, \{ 1, 8, 9 \}>,$\\
    $<x_0 + 101706993/27299960x_1 + 18966920067317/1396556753760x_2 - 52566520947228959/661502382364320x_3, \{ 3, 7, 9 \}>,$\\
    $<x_0 + 8357779/13676000x_1 + 832012900721/53816112000x_2 - 52566520947228959/331381678992000x_3, \{ 4, 7, 9 \}>,$\\
    $<x_0 + 15315551/3941730x_1 + 1805128839643/134428759920x_2 - 4778774631566269/63674422615440x_3, \{ 1, 6, 9 \}>,$\\
    $<x_0 + 157182551/19803111x_1 + 8357709171157/916567189524x_2 + 52566520947228959/1302441976313604x_3, \{ 2, 8, 9 \}>,$\\
    $<x_0 + 934130464082/28195297665x_1 - 722424939794/17170936277985x_2 + 1734695191258555647/2711100050112965x_3, \{ 3, 4, 9 \}>,$

             \end{tabular}};
        \end{tikzpicture}

                \begin{tikzpicture}
\node[anchor =center, scale=0.016cm] (n1) {\begin{tabular}{l}

    $<x_0 + 321930791/40882824x_1 + 2077424745499/232377971616x_2 + 119469365789156725/2971881878997024x_3, \{ 4, 6, 9 \}>,$\\
    $<x_0 + 2348995899/232715990x_1 + 79878409266713/11904819184440x_2 + 578231730419518549/5638916020363080x_3, \{ 4, 8, 9 \}>,$\\
    $<x_0 + 265051469/76851555x_1 + 17632972192729/1310472715860x_2 - 4778774631566269/56429749370820x_3, \{ 1, 2, 9 \}>,$\\
    $<x_0 + 354279823/65360760x_1 + 4487677961381/371510559840x_2 - 52566520947228959/1583749516597920x_3, \{ 5, 8, 9 \}>,$\\
    $<x_0 + 39094004/6221265x_1 + 4693061653729/424340043120x_2 - 4778774631566269/602987201273520x_3, \{ 6, 8, 9 \}>,$\\
    $<x_0 + 1544118371/363157215x_1 + 81790912745929/6192556830180x_2 - 578231730419518549/8799623255685780x_3, \{ 5, 7, 9 \}>,$\\
    $<x_0 + 114067582/37613829x_1 + 133555799743571/9162728744400x_2 - 4778774631566269/47346318348336x_3, \{ 1, 7, 9 \}>,$\\
    $<x_0 + 2681227603/214685310x_1 + 34879744787329/7321627812240x_2 + 578231730419518549/3468011040397680x_3, \{ 2, 5, 9 \}>,$\\
    $<x_0 + 1268226547/202228590x_1 + 84656286129721/6896803833360x_2 - 52566520947228959/3266786082401520x_3, \{ 7, 8, 9 \}>,$\\
    $<x_0 + 4122263849/653914755x_1 + 127540146188201/11150554402260x_2 - 52566520947228959/5281645935203820x_3, \{ 5, 6, 9 \}>,$\\
    $<x_0 + 212507/35389x_1 + 9659143653535/952249193784x_2 - 52566520947228959/6765730521835320x_3, \{ 2, 4, 9 \}>,$\\
    $<x_0 + 1444180727/74832840x_1 + 1615306935449/425349862560x_2 + 578231730419518549/1813266464093280x_3, \{ 2, 7, 9 \}>,$\\
    $<x_0 + 1662916511/296702670x_1 + 59231178794063/5059373928840x_2 - 52566520947228959/1960737368967720x_3, \{ 1, 3, 9 \}>,$\\
    $<x_0 + 1630742579/374235850x_1 + 253777536879103/19144409142600x_2 - 578231730419518549/9068068463878200x_3, \{ 3, 8, 9 \}>,$\\
    $<x_0 + 8387431/1582320x_1 + 349856118799/26981720640x_2 - 4778774631566269/115023075088320x_3, \{ 3, 6, 9 \}>,$\\
    $<x_0 + 121548566/18598245x_1 + 772339040311/60407099760x_2 - 52566520947228959/3862731994153200x_3, \{ 3, 5, 9 \}>,$\\
    $<x_0 + 148638/133405x_1 + 761476658647/51865942968x_2 - 52566520947228959/368507524787640x_3, \{ 2, 3, 9 \}>,$\\
    $<x_0 + 243818093/30689424x_1 + 4622877758687/523316058048x_2 + 52566520947228959/1239386864143680x_3, \{ 2, 6, 9 \}>,$\\
    $<x_0 + 237245453/81669390x_1 + 2104936287263/154736270920x_2 - 52566520947228959/539706046035240x_3, \{ 1, 4, 9 \}>,$\\
    $<x_0 + 1090343986/209317425x_1 + 3760712011622/297440060925x_2 - 52566520947228959/1267986979723275x_3, \{ 6, 7, 9 \}>,$\\
    $<x_0 + 263078731/58111680x_1 + 221037218393533/16845646245120x_2 - 14336323894698807/241794578932480x_3, \{ 1, 5, 9 \}>.$

       \end{tabular}};
              \end{tikzpicture}

\end{example}

\begin{remark}
Assume that $\calC$ lies on a smooth quadric. Once we have the equation of the curve we can compute the equations of tritangents \cite[Algorithm 3.1]{CelKulRenSay2018}. 

However, since we know how to obtain the complete level structure for the curves lying on a quadric cone, we could deform numerically a complete 2-level structure of an initial curve lying on a quadric cone to a target curve, namely the curve we want to have a complete 2-level structure, lying on smooth quadric by using the mathematical software Bertini \cite{BHSW06}. For a related work, see \cite[Section 4]{HauKulSerShe}.  
\end{remark}

\subsection{An Explicit Computation}\label{Test}

In this section, we verify Algorithm~\ref{AlgTheCon} on Example~\ref{Ex}. On one hand, we compute $\left(\frac{\vartheta[p_1]}{\vartheta[p_2]} \right)^4$ for $p_1=\{1,2,3,4,5\}$ and $p_2=\{3,4,5,6,9\}$ with the algorithm. The value is  

  \begin{align*}
  \left(\frac{\vartheta[p_1]}{\vartheta[p_2]}\right)^4&=388285435266921829/1618395584522100000 \\ 
                                                                               &\approx{\color{red}{0.239919979379812499393102579095}}31044756875610414688.
  \end{align*}

On the other hand,  we compute the following Riemann matrix of $\calC$ in Maple~\cite{Maple}

$$\begin{bmatrix}
1.07847i & -0.19708i & 0.30983 & 0.50267i \\
-0.19708i & 1.16996i  &  0.05607 & 0.24922i \\
0.30983 & 0.05607 & 1.23052i & -0.16325 \\
0.50267i & 0.24922i & -0.16325 & 1.42766i 
\end{bmatrix}.$$
Now, we can compute all the theta constants numerically as the values of theta functions at zeros. Hence, by looking through all the fourth powers of the quotients of them, we can determine that one of such values is approximately 
 \begin{align*}
&{\color{red}{0.239919979379812499393102579095}}95601233140655714802 + \\
&3.715929853910080160263726032046764685890634691202\times 10^{-53}i,  \end{align*}
which is strongly approximate to what we computed with the algorithm.

\bibliographystyle{plain}
%\bibliography{bibli}

\markboth{Bibliography}{}

\begin{thebibliography}{10}

\bibitem{Maple}
Maple.
\newblock Maplesoft, a division of Waterloo Maple Inc., Waterloo, Ontario.

\bibitem{ArbCorGriHar1985}
E.~Arbarello, M.~Cornalba, P.~Griffiths, and J.~D. Harris.
\newblock {\em Geometry of Algebraic Curves}, volume~1 of {\em 267}.
\newblock Springer-Verlag, New York, 1 edition, 1985.

\bibitem{BHSW06}
D.~J. Bates, J.~D. Hauenstein, A.~J. Sommese, and C.~W. Wampler.
\newblock {Bertini: Software for Numerical Algebraic Geometry}.
\newblock {Available at bertini.nd.edu with permanent doi:
  dx.doi.org/10.7274/R0H41PB5}.

\bibitem{BobKle}
A.~I. Bobenko and C.~Klein.
\newblock {\em Computational Approach to Riemann Surfaces}.
\newblock Springer, Berlin, Heidelberg, 2011.

\bibitem{magma}
W.~Bosma, J.~Cannon, and C.~Playoust.
\newblock The {M}agma algebra system {I}: {T}he user language.
\newblock {\em Journal of Symbolic Computation}, 24(3-4):235--265, 1997.
\newblock Computational algebra and number theory (London, 1993).

\bibitem{Tur2018}
T.~O. Celik.
\newblock {\em Propri{\'e}t{\'e}s g{\'e}om{\'e}triques et arithm{\'e}tiques
  explicites des courbes}.
\newblock PhD thesis, Universit{\`e} de Rennes 1, IRMAR, 2018.

\bibitem{CelKulRenSay2018}
T.~O. Celik, A.~Kulkarni, Y.~Ren, and M.~Sayyary~Namin.
\newblock Tritangents and their space sextics.
\newblock arXiv:1805.11702.

\bibitem{Dol2012}
I.~V. Dolgachev.
\newblock {\em Classical Algebraic Geometry: A Modern View}.
\newblock Cambridge University Press, 2012.

\bibitem{FarGruMan2017}
H.~Farkas, S.~Grushevsky, and R.~Salvati~Manni.
\newblock An explicit solution to the weak {S}chottky problem.
\newblock \url{https://arxiv.org/abs/1710.02938}, 2017.

\bibitem{GroHar2018}
B.~H.~Gross and J.~Harris.
\newblock On some geometric constructions related to theta characteristics.
\newblock In {\em Contributions to Automorphic Forms, Geometry, and Number}.
  Johns Hopkins University Press, 2004.

\bibitem{Har1982}
J.~Harris.
\newblock Theta-characteristics on algebraic curves.
\newblock {\em Transactions of the American Mathematical Society},
  271(2):611--638, 1982.

\bibitem{HauKulSerShe}
Jonathan~D. Hauenstein, Avinash Kulkarni, Emre~C. Sert{\"o}z, and Samantha~N.
  Sherman.
\newblock Certifying reality of projections.
\newblock In James~H. Davenport, Manuel Kauers, George Labahn, and Josef Urban,
  editors, {\em Mathematical Software -- ICMS 2018}, pages 200--208, Cham,
  2018. Springer International Publishing.

\bibitem{LubRob2016}
D.~Lubicz and D.~Robert.
\newblock Arithmetic on abelian and {K}ummer varieties.
\newblock {\em Finite Fields and Their Applications}, 39:130 -- 158, 2016.

\bibitem{Man1974}
Yu.~I. Manin.
\newblock {\em Cubic forms; Algebra, Geometry, Arithmetic}.
\newblock Elsevier, 2 edition, 1989.

\bibitem{Mum1966}
D.~Mumford.
\newblock On the equations defining abelian varieties. {I}.
\newblock {\em Inventiones mathematicae}, 1:287--354, 1966.

\bibitem{NarRit2017}
E.~Nart and C.~Ritzenthaler.
\newblock A new proof of a {T}homae-like formula for non hyperelliptic genus 3
  curves.
\newblock In {\em Arithmetic, geometry, cryptography and coding theory}, volume
  686 of {\em Contemporary Mathematics}, pages 137--155. American Mathematical
  Society, 2017.

\bibitem{RauFar1974}
H.~E. Rauch and H.~M. Farkas.
\newblock {\em Theta functions with applications to Riemann surfaces}.
\newblock Baltimore : Williams {\&} Wilkins, 1974.

\bibitem{Rie1857}
B.~Riemann.
\newblock Theorie der {A}belschen {F}unctionen.
\newblock {\em Journal f{\"ur} die reine und angewandte Mathematik},
  54:101--155, 1857.

\bibitem{Rie1866}
B.~Riemann.
\newblock {\"U}ber das {V}erschwinden der {$\vartheta$}-{F}unctionen.
\newblock {\em Journal f{\"ur} die reine und angewandte Mathematik},
  65:161--172, 1866.

\bibitem{Rit2004}
C.~Ritzenthaler.
\newblock Point counting on genus 3 non hyperelliptic curves.
\newblock In {\em Algorithmic Number Theory Symposium {(ANTS)}}, 2004.

\bibitem{Tho1870}
J.~Thomae.
\newblock {Beitrag zur Bestimmung von $\vartheta(0, 0, . . . 0)$ durch die
  Klassenmoduln algebraischer Funktionen}.
\newblock {\em Journal f{\"ur} die reine und angewandte Mathematik},
  71:201--222, 1870.

\bibitem{Web1876}
H.~Weber.
\newblock {\em Theorie der {A}belschen Funktionen vom Geschlecht 3}.
\newblock 1876.

\bibitem{Wen2003}
A.~Weng.
\newblock Constructing hyperelliptic curves of genus 2 suitable for
  cryptography.
\newblock {\em Mathematics of Computation}, 72(435-458), 2003.

\bibitem{Zarhin2008}
Yu.~G. Zarhin.
\newblock Del {P}ezzo surfaces of degree 1 and {J}acobians.
\newblock {\em Mathematische Annalen}, 340(2):407--435, 2008.

\end{thebibliography}

\end{document}